\documentclass{aptpub}

\authornames{D. Dolgopyat, P. Hebbar, L. Koralov, M. Perlman} 
\shorttitle{Multi-type branching processes with time-dependent branching rates} 

\numberwithin{equation}{section}

\usepackage[left=0.8in, right=0.8in, top=1.3in, bottom=1.5in]{geometry}

\def\eps{{\varepsilon}}

\def\EXP{\mathbb{E}}

\def\PROB{\mathbb{P}}

\def\bbK{{\mathbb{K}}}

\def\brcK{{\bar\cK}}

\def\breps{{\bar\varepsilon}}

\def\cK{\mathcal{K}}

\def\cP{\mathcal{P}}

\def\cX{\mathcal{X}}

\def\fb{{\mathfrak{b}}}

\def\tP{{\tilde P}}

\def\tZ{{\tilde Z}}

\def\tlambda{{\tilde\lambda}}
\def\tLambda{{\tilde\Lambda}}

\def\DS{\displaystyle}

\def\EXP{{\mathbb{E}}}

\begin{document}

\title{Multi-type branching processes with time-dependent branching rates}

\authorone[Univeristy of Maryland]{D. Dolgopyat} 


\authorone[Univeristy of Maryland]{P. Hebbar}


\authorone[Univeristy of Maryland]{L. Koralov}


\authortwo[Stanford Univeristy]{M. Perlman}

\addressone{Dept. of Mathematics, University of Maryland,
College Park, MD 20742}

\addresstwo{Dept of Mathematics, Stanford University,
Stanford, CA 94305}


\begin{abstract}
Under mild non-degeneracy assumptions on branching rates in each generation, we provide a criterion for almost-sure extinction of
a multi-type branching process with time-dependent branching rates. We also provide a criterion
for the total number of particles (conditioned on survival and divided by the expectation of the resulting random variable) to approach an exponential random variable as time goes to infinity.
\end{abstract}

   

   \section{Introduction} \label{intro}
   
   Mathematical study of branching processes goes back to the work of Galton and Watson \cite{WG} who were interested in the
   probabilities of long-term survival of family names. Later it was realized that similar mathematical models could be used to
   describe the evolution of a variety of biological populations, in genetics \cite{Fi1, Fi2, Fi3, Hal}, and in the study of certain
   chemical and nuclear reactions \cite{Se, HaU}. Branching processes are central in the
   study of evolution of various populations such as bacteria, cancer cells, carriers of a particular form of a gene, etc., where each
   member of the population may die or produce offspring independently of the rest.

    The individuals involved in the process are referred to as particles. In many models, the particles may be of different types, representing individuals with different characteristics. For example, in epidemiology, multi-type continuous time Markov branching process may be used to describe the dynamics of the spread of parasites of two types that can mutate into each other in a common host population \cite{BDR}; when modeling cancer, particles of different types may represent cells that have accumulated different numbers of mutations \cite{Du}; 
   in physics, cosmic ray cascades, which involve electrons producing photons and photons producing electrons,
   can be modeled by a 2-type branching process \cite{Mode}. 
   In addition, a vast number of applications of multi-type branching processes 
   in biology can be found in \cite{HJV, KA}.

   The current paper concerns the long-time behavior of multi-type branching processes with time-dependent branching rates. Let us stress that
   the temporal inhomogeneity is due to the dependence of the branching rates not on the ages of the particles (which is a well-studied model), but on time (this dependence may model a varying environment for the entire process).  We believe that the methods of our paper could be used to handle more general models
   such as those where, in addition to time dependence, the branching rate may depend on the age of the particles and/or on
   their spatial location if the spatial motion in a bounded domain is allowed. This may be a subject of future work.

   For multi-type processes with constant branching rates, according to classical results (see Chapter 5 of \cite{AN} and references therein), three
   different cases can be distinguished. In the super-critical case, the expectation of the total population size
   grows exponentially, and the total population grows exponentially with positive probability as time goes to infinity.
   In the sub-critical case, the expectation of the total population size decays exponentially, and the population goes extinct with overwhelming probability, i.e., the probability that the population at time $n$ is non-zero decays exponentially in $n$. In the critical case, the population also goes extinct, but the expectation of the total population size remains bounded away from zero and infinity, and the probability of survival decays as $c/n$ for some
   $c > 0$. Moreover, after conditioning on survival,  the size of the population divided by its expectation tends to an exponential random variable. Whether the process is super-critical, sub-critical, or critical, can be easily determined by examining the (constant) branching rates.
   
   The question we address in the case of time-dependent branching rates is how to distinguish between different kinds of asymptotic behavior of the process based on the behavior of the branching rates. Our first result gives a criterion for almost sure extinction of the process in terms of the asymptotic behavior of the branching rates, under mild non-degeneracy assumptions on the branching rates at each time step. In the case of single-type branching processes, a similar result was obtained by Agresti \cite{Agr}. An earlier partial result in this direction (for single-type branching processes) was obtained by Jagers \cite{Jag}, who also provided a sufficient condition for the exponential limit (in distribution) of the size of the population (after conditioning on survival and dividing
   by the expectation of the resulting random variable). 
   For single type branching processes, a necessary and sufficient condition for
   exponential distribution of particle number conditioned on survival
   in terms of the branching rates was obtained independently in \cite{BP, Ker}.
   Our second result gives a necessary and sufficient condition for the existence of such an exponential limit in the case of multi-type branching processes.

   Based on our results, it is natural classify all the branching processes with time-dependent branching rates (under the non-degeneracy assumptions) into three categories, based on their asymptotic behavior. Processes in the first category (which includes super-critical processes with time-independent rates), are distinguished by a positive probability of survival for infinite time. Processes in the second category (which includes critical processes with time-independent rates) go extinct with probability one, and the size of the population, after conditioning on survival and normalization, tends to the exponential limit. Processes in the third category 
	(which includes sub-critical processes
   with time-independent rates) go extinct with probability one, but do not have the exponential limit. 
   
   It should be stressed that, in contrast to the case of time-independent rates (when the expected population size either grows exponentially, decays exponentially, 
   or is asymptotically constant), now the expected population size may fluctuate greatly in each of the cases, which  makes the analysis more complicated. 
   
   Let us also remark that some of the classical results on the asymptotic behavior of branching processes in the time-independent case carry over to the 
   case at hand, while others do not. For example, in the time-independent case, super-critical processes have the property that the process normalized by  expected population size tends to a random limit. An analogue of this statement still holds in the case of time-dependent branching rates, 
   as follows from the results of \cite{Jo}. Further results on $L^p$ and almost sure convergence,
   including those in the case of countably many 
   particle types, can be found in \cite{Big}.
   Sufficient conditions for the continuity of the limiting distribution function were given in \cite{Co}.

	On the other hand, in the time-independent case, a sub-critical process conditioned on survival tends to a random limit. 
        Now, our processes in the third category do not necessarily have this property
        (e.g., the population, conditioned on survival, may grow along a subsequence). A more
        detailed analysis of the near-critical behavior of processes with time-dependent rates will
        be the subject of a subsequent paper. 
   
        In the next section, we introduce the relevant notation and formulate the main results.
        The proofs are presented in Sections~\ref{settt} and~\ref{sec3}.
   In Section~\ref{brct}, we briefly discuss an application of our results to the case of continuous time branching.

   \section{Notation and statement of main results} \label{nore}

Let $S = \{1,...,d\}$ be the set of possible particles types. Suppose that for each
$i \in S$ and $n \geq 0$ there is a distribution $P_n(i,\cdot)$ on $\mathbb{Z}_+^d$. For $a = (a_1,...,a_d) \in \mathbb{Z}_+^d$, 
$P_n(i,a)$ represents the probability that a particle of type $i$ that is alive at time $n$ is replaced in the next generation by $a_1+...+a_d$ particles: $a_1$ particles of type one, $a_2$ particles of type two, etc. 
A $d$-type branching process $\{Z_n\}$ is obtained by starting with a 
positive finite number of particles at time zero, and then replacing each particle of each type $i$, $i \in S$, that is alive at time $n$, $n \geq 0$, by particles of various types according to the distribution $P_n(i, \cdot)$ independently of the other particles alive at time $n$ and of the past, thus obtaining the population at time $n+1$.

We write $Z_n = (Z_n(1),...,Z_n(d))$, where $Z_n(i)$ is the number of particles of type $i$ at time $n$. When the initial population consists of one particle of 
type $j$, we may write $_jZ_n(i)$ to represent the number of particles of type $i$ at time $n$.
Thus $\mathbb{E}(_jZ_n(i))$ means the same as $\mathbb{E}(Z_n(i)\big|Z_0 = e_j)$, where $e_j$ is the unit vector in the $j$-th direction. Let 
$_jX_n$ denote a generic random vector with distribution $P_n(j,\cdot)$.

For $s = (s_1,...,s_d) \in [0,1]^d$, let 
\[
  f_n^j(s) = \mathbb{E}\Big(\prod_{i = 1}^{d} s_i^{Z_n(i)} | Z_0 = e_j\Big),
\]
\[
  g_n^j(s) = \mathbb{E}\Big(\prod_{i = 1}^{d} s_i^{Z_{n+1}(i)} | Z_n = e_j\Big).
\]
At times, we may drop the superscript from either of those expressions, and then $f_n(s)$ and $g_n(s)$ become vectors. Note that
\[
 f_n(s) = f_{n-1}(g_{n-1}(s)) = (g_0 \circ g_1 \circ ... \circ g_{n-1})(s),  \text{ and } f_n(\mathbf{1}) = \mathbf{1}
\]
where $\mathbf{1} = (1,...,1)$. We also define
\[
f_{k,n}(s) = (g_k \circ ... \circ g_{n-1})(s).
\]
Thus $f_{0,n} = f_n$.  We denote
 \[
M_n(j,i) = \frac{\partial f_n^j}{\partial s_i}(\mathbf{1}) = \mathbb{E}(Z_n(i)| Z_0 = e_j),
 \]
\[
 A_n(j,i) = \frac{\partial g_n^j}{\partial s_i}(\mathbf{1}) = \mathbb{E}(Z_{n+1}(i) |Z_n = e_j),
\]
Then,
\[
 M_n = A_0 ....A_{n-1},
\]
where $A_n$ and $M_n$ are viewed as matrices. Also define
\[
M_{k,n} = A_k ... A_{n-1}.
\] 

Let $\| \cdot \|$ denote the following norm of a $d$-dimensional vector: $\| v \| = |v_1| + ... + |v_d|$. 
We will use certain non-degeneracy assumptions on the distribution of descendants at each step. We assume that there are $\varepsilon_0, K_0 > 0$ such that for all $i,j \in S$ the following bounds hold.
\begin{enumerate}
 \item $\mathbb{P}(Z_{n+1}(i) \geq 2 | Z_n = e_j) \geq \varepsilon_0$.
 \item $\mathbb{P}(Z_{n+1} = \mathbf{0} | Z_n = e_j) \geq \varepsilon_0$.
 \item $\mathbb{E}(\|Z_{n+1}\|^2 \big| Z_n = e_j) \leq K_0$.
\end{enumerate}

The following proposition is a generalization of the Perron-Frobenius theorem to the case when the positive matrices forming a product are allowed to be distinct. 

\begin{prop} \label{Propabc} 
Under Assumptions  1 and 3, 
there are two sequences of vectors $v_n, u_n \in \mathbb{R}^d$, $n \geq 0$, such that

(a) $\|u_n\| = \|v_n\| = 1$.

(b) $v_n(i), u_n(i) \geq \bar{\varepsilon}$, for some $\bar{\varepsilon} > 0$ and all $n \geq 0$, $i \in S$,

(c) There are sequences of positive numbers $\lambda _n$ and $\tilde{\lambda}_{n}$ and a positive constant $a$ such that $\lambda_n, \tilde{\lambda}_{n} \in (a^{-1}, a)$ for $n \geq  0$ and
\[
A_{n-1}v_n = \lambda_{n-1} v_{n-1} ,\hspace{1cm} A_{n-1}^{T} u_{n-1} = \tilde{\lambda}_{n-1} u_n. 
\]

(d) For each $\delta > 0$ there is $k' \in \mathbb{N}$ such that 
\[
(1-\delta)v_n \leq \frac{M_{n, n+k} v}{\|M_{n, n+k} v\|} \leq (1+\delta)v_n,~~~~(1-\delta)u_{n+k} \leq \frac{M^T_{n, n+k} u}{\|M^T_{n, n+k} u\|} \leq (1+\delta)u_{n+k}
\]
whenever $k \geq k'$, $v$ and $u$ are non-zero vectors with non-negative components, and inequality between vectors is understood as the inequality 
between their components.

(e) There is $K > 0$ 
such that if we define $\Lambda_n = \prod_{i = 0}^{n-1}\lambda_i$ and $\tilde{\Lambda}_n = \prod_{i = 0}^{n-1}\tilde{\lambda}_i$,
then
\[
\frac{1}{K} \leq \frac{\Lambda_n}{\tilde{\Lambda}_n} \leq K,~~~~~\frac{1}{K} \leq \frac{M_{k,n}(j,i)}{{(\Lambda}_n/\Lambda_k)} \leq K,~~j,i \in S.
\]
\end{prop}

This proposition can be derived from the results of Chapter 3 of \cite{Sen}, for example. 
 Indeed, from our Assumptions 1-3, it follows
that the matrices $A_n$ have the Birkhoff's contraction coefficient (in the terminology of \cite{Sen}) 
uniformly bounded away from one. This implies that the conditions of
Lemma 3.4. of \cite{Sen} are met (which, in particular, implies that the family $M_{k,n}$ is weakly ergodic (see \cite{Sen}). This lemma and Exercise 3.5 of \cite{Sen}
easily imply the existence of vectors $u_n$ and $v_n$. Their required properties are also not difficult to establish.
For the sake of completeness we provide an independent  
proof in Appendix~\ref{AppPF}.

\begin{rem} The vectors $v_n$ and the numbers $\lambda_n$ are uniquely defined by the above conditions, 
  as seen from the proof of the Proposition. The vectors $u_n$ and the numbers
  $\tilde{\lambda}_n$ will be defined uniquely by specifying $u_0$, which we assume to be fixed as an arbitrary vector satisfying conditions (a) and (b).
\end{rem}

The probabilistic meaning of vectors $u_n$ and $v_n$ is the following. The vector $u_n$ gives the asymptotic proportions 
of different particles in the population provided that $Z_n$ is large 
(see \eqref{colli}, \eqref{colli2} 
for the precise statement).
To see the meaning of $v_n$, consider  the total number of particles at time $n$,
$z_n^*=\langle Z_n, \mathbf{1} \rangle.$ It will be apparent from the proof of Proposition \ref{Propabc} 
that 
$$ \lim_{N\to\infty} \frac{\EXP(z_N^*|Z_n=u')}{\EXP(z_N^*|Z_n=u'')}=
\frac{\langle v_n, u'\rangle}{\langle v_n, u''\rangle}
$$
for each $u', u''\in \mathbb{Z}^d$ and each $n\in \mathbb{Z}_+$.
Thus $v_n$ controls the expected future size of the population.

Our first result gives a necessary and sufficient condition for the almost sure extinction of $\{Z_n\}$. 
\begin{thm} \label{yth}
Under Assumptions 1-3,  if extinction  of the process $\{Z_n\}$ occurs with probability one for some initial population, then
$\sum_{k=1}^{\infty}({1}/{\Lambda}_{k}) = \infty$. If $\sum_{k=1}^{\infty}({1}/{\Lambda}_{k}) = \infty$, then extinction with probability one occurs for every initial population. 
\end{thm}

\begin{rem} Here and below, when we talk about initial population, we mean
that $Z_0=u$ for some deterministic vector $u.$ 
\end{rem}

\begin{rem} The first statement of the theorem can be deduced from the results of \cite{Jo}.  In fact, the assumptions needed for the first
part are weaker than our assumptions above. For example, weak ergodicity (see \cite{Jo}) is sufficient. However, the assumption that the matricies $A_k$
are uniformly bounded from below plays an important role in our proof of the second statement, as well as in the proof of Theorem~\ref{the2} below. 
We note that finding the least restrictive conditions for the validity of Theorems~\ref{yth} and \ref{the2} remains an interesting open problem. We refer
the reader to the paper \cite{Ker} for recent results in the case of single-type branching processes. 
\end{rem} 

 The following lemma easily follows from Theorem~\ref{yth}. 
\begin{lem} \label{ljj3}
	Suppose Assumptions 1-3 hold.
	
(a) Given $l\geq 0$, consider the process $\{_j Z'_n\}$ that starts with one particle of type $j$ alive at time 
$l$ followed by branching with the distributions  
$P_{l}, P_{l+1}, \dots$ (where the distributions $P_i$ are used in the definition of the branching process in the beginning of the section).
Extinction for this process occurs with probability one if and only 
if $\sum_{k=1}^{\infty}({1}/{\Lambda}_{k}) = \infty$.  

(b) Given $l>0$, the extinction  of $\{Z_n\}$ (or, equivalently, $\{_j Z'_n\})$  occurs with probability one if and only if
$\DS \sum_{k=1}^\infty \frac{1}{\Lambda_{lk}}=\infty.$
\end{lem}

\begin{rem}
	The divergence of $\sum_{k=1}^\infty \frac{1}{\Lambda_{lk}}$ is the extinction condition for the process
	$\{Z_{ln}\}$ obtained by observing our process only at the moments of time that are multiples of $l.$ 
\end{rem}  

\begin{proof}[Proof of Lemma \ref{ljj3}] For part (a)  it suffices to note that $\DS \sum_{k=1}^{\infty} \frac{1}{\Lambda_k} = \infty$ if and only if $\DS \sum_{k=l+1}^{\infty} \frac{\Lambda_{l}}{\Lambda_k} = \infty$,
while the latter is equivalent to the almost sure extinction of the process $\{_j Z'_n\}$ by Theorem~\ref{yth}. 

To prove part (b), we observe that, under Assumption 3, there exists a constant $C$ such that for each
$k$ and each $lk\leq n<l(k+1)$ we have
$\DS \frac{\Lambda_{lk}}{C} \leq \Lambda_n \leq C\Lambda_{lk}. $
\end{proof}

The following lemma will be derived in the end of the next section  using the formulas encountered in the proof of Theorem~\ref{yth}. 
\begin{lem} \label{elesst} Under Assumptions 1-3, for each initial population of the branching process, there is a constant $C > 0$ such that
\begin{equation} \label{qwq1}
\frac{\Lambda_n}{C}  \leq  \mathbb{E} \| Z_n \| \leq C  \Lambda_n,~~n \geq 1,
\end{equation} 
\begin{equation} \label{qwq2}
\frac{1}{C} \left(\sum_{k=1}^{n} \frac{1}{\Lambda}_{k}\right)^{-1} \leq \mathbb{P} (Z_n \neq \mathbf{0}) 
\leq C \left(\sum_{k=1}^{n} \frac{1}{\Lambda}_{k}\right)^{-1},~~n \geq 1.
\end{equation}
\end{lem}

To formulate the next theorem, we will make use of the following assumptions:
\\

4. The random variables $\|_jX_n\|^2$, $j \in S$, $n \geq 0$,  are uniformly integrable. 
\\

5. $\mathbb{P} (Z_n \neq \mathbf{0}) \rightarrow 0$ as $n \rightarrow \infty$ (equivalently, $\sum_{k=1}^{n}({1}/{\Lambda}_{k}) \rightarrow \infty$, by (\ref{qwq2})). 
\\

6. $\mathbb{E} \| Z_n \| /\mathbb{P} (Z_n \neq \mathbf{0}) \rightarrow \infty $ as $n \rightarrow \infty$ (equivalently, ${{\Lambda}_{n}}\sum_{k=1}^{n}({1}/{\Lambda}_{k}) \rightarrow \infty$, by (\ref{qwq1}), (\ref{qwq2})).
\\

Let $\zeta_n = (\zeta_n(1),...,\zeta_n(d))$ be the random vector obtained from $Z_n$ by conditioning on the event that $Z_n \neq \mathbf{0}$.
In other words, we treat the event $Z_n \neq \mathbf{0}$ as a new probability space, with the measure $ \mathbb{P}'$ obtained from the underlying  measure $\mathbb{P}$ via $\mathbb{P}'(A) = \mathbb{P}(A)/ \mathbb{P}(Z_n \neq \mathbf{0})$.  When we write $_j\zeta_n$, we mean that the initial population for the branching process is specified as $e_j$.

We will prove exponential limit for the multi-type random variable under the assumptions listed above.
\begin{thm} \label{the2}
Under Assumptions 1-6, for each initial population of the branching process and each vector $u$ with positive components,  we have the following limit in distribution
\begin{equation} \label{maineq}
\frac{\langle \zeta_n,u \rangle}{  \mathbb{E} \langle \zeta_n,u \rangle } \rightarrow \xi, ~~{\it as}~n \rightarrow \infty,
\end{equation}
where $\xi$ is an exponential random variable with parameter one. Moreover, if  Assumptions~1-5  are satisfied and, for some initial population, 
the limit in (\ref{maineq}) is as specified, then Assumption~ 6 is  also satisfied. 
\end{thm}

We say that a process is {\em uniformly critical} if it satisfies
Assumptions 1-4 and 
there is a constant $b$ such that
for each $n,k,$ $i,j$, we have
\begin{equation}
\label{UniCrit}
\frac{1}{b}\leq M_{n, n+k}(j,i) \leq b.
\end{equation}

For uniformly critical processes, $\Lambda_k$ are uniformly bounded from above and below, so 
we have
$$\sum_{k=1}^\infty \frac{1}{\Lambda_k}=\infty, \quad
\lim_{n\to \infty} \Lambda_n \left(\sum_{k=1}^n \frac{1}{\Lambda_k}\right)=\infty. 
$$
Therefore, uniformly critical processes become extinct with probability 1 and 
the distribution of the appropriately scaled number of particles  at time $n,$
conditioned on survival, converges to exponential.

 The next proposition and the lemma that follows will be helpful for comparing our results to those of \cite{Jag}. 
An important part of Proposition~\ref{PrRarify} (part (d)) shows that, under \eqref{UniCrit}  (or even under a weaker condition (\ref{NotSuperCrit})),  Assumption 2 almost follows from Assumption 1, in the sense that Assumption 2 is satisfied for an appropriate subprocess. Given $l,$ let $\tP_n$ be
the transition probability of the process $\{\tZ_{n}\}$, where $\tZ_n=Z_{nl}.$ That is,
$\tP_n(i,a)$ represents the probability that a particle of type $i$ that is alive at time $nl$ is replaced in  generation $(n+1)l$ by $a_1+...+a_d$ particles: $a_1$ particles of type one, $a_2$ particles of type two, etc. 

\begin{prop}
\label{PrRarify}
(a) If $P_n$ satisfies Assumption 1, then $\tP_n$ satisfies Assumption 1 for each $l$.

(b) If $P_n$ satisfies Assumption 3, then $\tP_n$ satisfies Assumption 3 for each $l$.

(c) If $P_n$ satisfies Assumption 4, then $\tP_n$ satisfies Assumption 4 for each $l$.

(d) If $P_n$ satisfies Assumption 1, and there is a constant $\fb$ such that
for each $n,k, j$
\begin{equation}
\label{NotSuperCrit}  
\mathbb{E}(|Z_{n+k}||Z_n=e_j)\leq \fb,
\end{equation}
then there exist $l=l(\eps_0, \fb)$ and $\eps_1=\eps_1(\eps_0, \fb)$ such that for each $n$ and $j$
$$ \tP_n(\tZ_{n+1}=\mathbf{0}|\tZ_{n}=e_j)\geq \eps_1.$$
\end{prop}

The above Proposition is proved in Appendix \ref{AppSkip}. The following lemma is proved in Section~\ref{sec3}.
\begin{lem}
\label{CrJagers}
If $\{Z_n\}$ satisfies Assumptions 1 and 4, and \eqref{UniCrit} holds, then extinction happens with
probability one and \eqref{maineq} holds.
\end{lem}

For single type branching processes, Lemma~\ref{CrJagers} is helpful in showing that our results imply  Theorem 5 of \cite{Jag}. 
In fact, the assumptions of  Theorem 5 of \cite{Jag} (generalized to the multi-type case) are: (a) our Assumption 4, 
(b) that \eqref{UniCrit} holds, and (c) that there is $\breps_0 > 0$ such that for each $n, i, j $ 
\begin{equation}
\label{JIn}
\mathbb{E} (Z^2_{n+1}(i) -Z_{n+1}(i)|Z_n=e_j)\geq \breps_0. 
\end{equation} 
We claim that under Assumption 4, \eqref{JIn} is equivalent to Assumption 1.
On the one hand,
$$ \mathbb{E} (Z^2_{n+1}(i) -Z_{n+1}(i)|Z_n=e_j)\geq 2 \mathbb{P}(Z_{n+1}(i)\geq 2|Z_n=e_j). $$
On the other hand,  by Assumption 4, we can take $N\geq 2$ such that
$$ \mathbb{E} (Z^2_{n+1}(i) \chi_{Z_{n+1}(i)\geq N} |Z_n=e_j)\leq \frac{\breps_0}{2} $$
for all $n$, 
where $\chi_{Z_{n+1}(i)\geq N}$ is the indicator function of the the event $\{Z_{n+1}(i)\geq N\}$. 
Then
$$ \mathbb{E} (Z^2_{n+1}(i) -Z_{n+1}(i)|Z_n=e_j)\leq \frac{\breps_0}{2}+(N^2-N) 
\mathbb{P}(Z_{n+1}(i)\geq 2|Z_n=e_j). $$ 
Thus if \eqref{JIn} holds, then
$$\mathbb{P}(Z_{n+1}(i)\geq 2|Z_n=e_j)\geq \frac{\breps_0}{2(N^2-N)}, $$
proving that Assumption 1 is equivalent to \eqref{JIn}. 

The results of  Theorem 5 of \cite{Jag} are: our Lemma~\ref{lemmac} (in the single type case) and
our formula \eqref{maineq} (in the single type case). The latter holds by Lemma \ref{CrJagers}. We prove
Lemma~\ref{lemmac} in Section~\ref{sec3} under Assumptions 1-6. However, under Assumptions 1, 4, and \eqref{UniCrit}, the conclusion of the
lemma still holds (the argument is similar to that in the proof of Lemma \ref{CrJagers}).

\section{Survival vs extinction} \label{settt}

\begin{proof}[Proof of Theorem~\ref{yth}]

(PART I) $\sum_{k=1}^{\infty}({1}/{\Lambda}_{k}) < \infty$ implies positive probability of survival.
\medskip

Let us fix $Z_0 = e_j$ with an arbitrary $j \in S$. Let ${\cal{F}}_n$ be the $\sigma$-algebra generated by the branching process $\{Z_n\} = \{_jZ_n\}$.
Let $z_n = \langle Z_n, v_n \rangle$. Then, 
\[
 \mathbb{E}(z_{n+1}|{\cal F}_n) = \langle \mathbb{E}(Z_{n+1}|Z_n), v_{n+1} \rangle = \langle A_n^TZ_n, v_{n+1} \rangle = \langle Z_n, A_nv_{n+1} \rangle = \lambda_n z_n
\]
Accordingly, $\{z_n/\Lambda_n\}$ is a positive martingale, and hence it converges to some random variable $z_{\infty}$.
Now let
\[
 D_n(j_1,j_2) = \text{Cov}(Z_n(j_1),Z_n(j_2)).
\]
One step analysis gives
\begin{equation}\label{equ}
 D_{n+1} = A_n^TD_nA_n + S_n,
\end{equation}
where 
\[
 S_n = \sum_{i = 1}^{d} M_n(j,i)\sigma_n^2(i)
\]
and 
\[
 \sigma_n^2(j_1,j_2)(i) = \text{Cov}(_iX_n(j_1), {_iX_n(j_2)}).
\]

By Proposition \ref{Propabc}, there exists a constant B such that, $\|S_n\| \leq B\Lambda_n$, where $\|\cdot\|$ is a matrix norm.
Iterating (\ref{equ}), we get
\[
 D_n = \sum_{k = 0}^{n-1}M_{k+1,n}^TS_kM_{k+1,n}.
\]
Hence,
\[
 \|D_n\| \leq B_1 \sum_{k = 0}^{n-1} \Big(\frac{\Lambda_n}{\Lambda_{k+1}}\Big)^2\Lambda_k \leq B_2 \Lambda_n^2 \sum_{k = 0}^{n-1}\frac{1}{\Lambda_k}
\]
with some constants $B_1$, $B_2$. 

Thus  $\|D_n\| \leq \tilde{B}\Lambda_n^2$, and so the martingale $\{z_n/\Lambda_n\}$ is uniformly bounded in $L^2$. Therefore,
 $\mathbb{E}(z_{\infty}) = \mathbb{E}(z_0) > 0$, and hence $\mathbb{P}(z_{\infty} > 0) > 0$, implying that the probability of survival of the branching process starting with a single particle of type $j$ is positive. Therefore, the probability of survival is positive for every initial population.\\

(PART II) $\sum_{k=1}^{\infty}({1}/{\Lambda}_{k}) = \infty$ implies that extinction occurs with probability one.\\

Recall that $f_{0,n}(s) = g_{0}(f_{1,n}(s))$ and $f_{0,n}(\mathbf{1}) = g_{0}(\mathbf{1}) = \mathbf{1}$. Determining the asymptotic behavior of  $ \langle \mathbf{1} - f_{0,n}(s) , u_{0} \rangle$ will be helpful for proving the theorem and also later in the proof of (\ref{ew1}). By the Taylor formula with respect to $s = \mathbf{1}$,
\[
 \langle \mathbf{1} - f_{0,n}(s) , u_{0} \rangle   = \langle Dg_{0}(\mathbf{1})(\mathbf{1} - f_{1,n}(s)) , u_{0} \rangle - \frac{1}{2}\langle(\mathbf{1} - f_{1,n}(s))^THg_{0}(\eta_{1,n})(\mathbf{1} - f_{1,n}(s)) , u_{0} \rangle
\]
\[
= \langle A_{0}(\mathbf{1} - f_{1,n}(s)) , u_{0} \rangle - \frac{1}{2}\langle(\mathbf{1} - f_{1,n}(s))^THg_{0}(\eta_{1,n})(\mathbf{1} - f_{1,n}(s)) , u_{0} \rangle,
\]
where $Dg_0$ is the gradient of $g_0$ and $\eta_{1,n} = \eta_{1,n} (j,s)$ satisfies  $f^j_{0,n}(s) \leq \eta_{1,n} \leq 1$ for each component $j \in S$ and $s \in [0,1]^d$. Here $H g_0 $ stands for the Hessian matrix applied to each component of the vector function $g_0$ separately, then multiplied by vectors $(\mathbf{1} - f_{1,n}(s))^T$ and $(\mathbf{1} - f_{1,n}(s))$ to
get scalars, which are then multiplied by the corresponding components of $u_0$ to form the scalar product.
Therefore, by taking the transpose of $A_0$,
\[
  \langle \mathbf{1} - f_{0,n}(s) , u_{0} \rangle  = \langle (\mathbf{1} - f_{1,n}(s)) , A_{0}^{T}u_{0} \rangle - \frac{1}{2}\langle(\mathbf{1} - f_{1,n}(s))^THg_{0}(\eta_{1,n})(\mathbf{1} - f_{1,n}(s)) , u_{0} \rangle
  \]
\[= \langle (\mathbf{1} - f_{1,n}(s)) , \tilde{\lambda}_{0}u_{1} \rangle - \frac{1}{2}\langle(\mathbf{1} - f_{1,n}(s))^THg_{0}(\eta_{1,n})(\mathbf{1} - f_{1,n}(s)) , u_{0} \rangle.
\]
Thus, for $s \neq \mathbf{1}$, 
\[
(\langle \mathbf{1} - f_{0,n}(s) , u_{0} \rangle) ^{-1}  = \Big(\langle (\mathbf{1} - f_{1,n}(s)) , \tilde{\lambda}_{0}u_{1} \rangle - \frac{1}{2}\langle(\mathbf{1} - f_{1,n}(s))^THg_{0}(\eta_{1,n})(\mathbf{1} - f_{1,n}(s)) , u_{0} \rangle\Big)^{-1}
\]
\[
= ( \tilde{\lambda}_{0}\langle (\mathbf{1} - f_{1,n}(s)) ,u_{1} \rangle)^{-1} \Big( 1  +\frac{\langle -\frac{1}{2}(\mathbf{1} - f_{1,n}(s))^THg_{0}(\eta_{1,n})(\mathbf{1} - f_{1,n}(s)) , u_{0} \rangle}{\tilde{\lambda}_{0}\langle (\mathbf{1} - f_{1,n}(s)) ,u_{1} \rangle}\Big)^{-1}
\]
\[
= \frac{1}{\tilde{\lambda}_{0}\langle (\mathbf{1} - f_{1,n}(s)) ,u_{1} \rangle} +\frac{\langle \frac{1}{2}(\mathbf{1} - f_{1,n}(s))^THg_{0}(\eta_{1,n})(\mathbf{1} - f_{1,n}(s)) , u_{0} \rangle}{\tilde{\lambda}_{0}\langle (\mathbf{1} - f_{1,n}(s)) ,u_{1} \rangle \langle \mathbf{1} - f_{0,n}(s) , u_{0} \rangle },
\]
where the last equality follows from the simple relation
\[
 \frac{1}{a} = \frac{1}{b}\left(1 - \frac{c}{b}\right)^{-1}~~\Longrightarrow~~\frac{1}{a} = \frac{1}{b} + \frac{c}{ba}.
\]
By iterating the previous equality $n$ times, we get
\begin{equation} \label{eq1z}
\langle \mathbf{1} - f_{0,n}(s) , u_{0} \rangle ^{-1}
= \frac{1}{\tilde{\Lambda}_{n}\langle \mathbf{1}-s , u_n\rangle} + \frac{1}{2} \sum_{k = 0}^{n-1} \frac{\langle (\mathbf{1} - f_{k+1,n}(s))^THg_{k}(\eta_{k+1,n})(\mathbf{1} - f_{k+1,n}(s)) , u_{k} \rangle}{\tilde{\Lambda}_{k+1}\langle (\mathbf{1} - f_{k+1,n}(s)) ,u_{k+1} \rangle \langle \mathbf{1} - f_{k,n}(s) , u_{k} \rangle },
\end{equation}
where $f^j_{k,n}(s) \leq \eta_{k+1,n}(j,s) \leq 1$ for each $k \geq 0 $ and $j \in S$.

Let 
\begin{equation} \label{lll}
 \alpha(n,s) = \frac{1}{2}\sum_{k = 0}^{n-1} \frac{\langle (\mathbf{1} - f_{k+1,n}(s))^THg_{k}(\eta_{k+1,n})(\mathbf{1} - f_{k+1,n}(s)) , u_{k} \rangle}{\tilde{\Lambda}_{k+1}\langle (\mathbf{1} - f_{k+1,n}(s)) ,u_{k+1} \rangle \langle \mathbf{1} - f_{k,n}(s) , u_{k} \rangle },
\end{equation}
where we note again that the dependence on $s$ also lies in the vector  $\eta_{k+1,n}$ since the components of $\eta_{k+1,n}$ satisfy $ f_{k,n}(s)(i) \leq \eta_{k+1,n}(i) \leq 1$. Then
(\ref{eq1z}) takes the form
\begin{equation} \label{hhk}
 \langle \mathbf{1} - f_{0,n}(s) , u_{0} \rangle  =  \Big(\frac{1}{\tilde{\Lambda}_{n}\langle \mathbf{1}-s , u_n\rangle} + \alpha(n,s)\Big)^{-1}.
\end{equation}

We will need the following lemma which will be proved after we complete the proof of this Theorem.

Let us denote 
\[
\Xi_n = \sum_{k = 0}^{n-1}\frac{1}{{\Lambda}_{k+1}}.
\]
These are the partial sums of the series found in Theorem~\ref{yth},
but with the index of summation shifted in order to make the arguments below more transparent.

\begin{lem} \label{lemma1}  Under Assumptions 1-3, 
there exists  $C > 1$  such that for each $n$ and each $s \in [0,1]^d  \setminus \{ \mathbf{1} \} $ we have  
\[
\frac{1}{C} \leq \frac{\alpha(n,s)}{\Xi_n} \leq C.
\] 
\end{lem}

Using Lemma \ref{lemma1} in equation  (\ref{hhk}), we get 
\[
 \langle \mathbf{1} - f_{0,n}(\mathbf{0}) , u_{0} \rangle  \leq  \Big(\frac{1}{\tilde{\Lambda}_{n}} +  \frac{\Xi_n}{C}  \Big)^{-1}.
\]
Therefore,
\[
 \langle \mathbf{1} - f_{0,n}(\mathbf{0}) , u_{0} \rangle \leq \frac{C}{\Xi_n}. 
\]
 We note that that $1 - f_{0,n}^j(\mathbf{0}) = \mathbb{P}(Z_n \neq \mathbf{0} | Z_0 = e_j)$ 
 for each $j \in S$, and hence,  
if $\lim_{n \rightarrow \infty}\Xi_n  = \infty$, then 
\[
 \lim_{n \rightarrow \infty}\mathbb{P}(Z_n \neq \mathbf{0} | Z_0 = e_j) = 0.
\]
Thus, extinction occurs with probability one if the initial population is $e_j$. Therefore, since $j$ was arbitrary, extinction occurs with probability one for every
initial population.
\end{proof} 
%

\begin{proof}[Proof of Lemma \ref{lemma1}]  
The statement will follow if we prove the following bounds on the terms in the sums for  
$\alpha(n,s)$ and $\Xi_n$: for each $0 \leq k \leq n-1$ and $s \in [0,1]^d  \setminus \{ \mathbf{1} \} $, we have 
\begin{equation} \label{ceq}
\frac{1}{ C \Lambda_{k+1}}
\leq \frac{\langle (\mathbf{1} - f_{k+1,n}(s))^THg_{k}(\eta_{k+1,n})(\mathbf{1} - f_{k+1,n}(s)) , u_{k} \rangle}{\tilde{\Lambda}_{k+1}\langle (\mathbf{1} - f_{k+1,n}(s)) ,u_{k+1} \rangle \langle \mathbf{1} - f_{k,n}(s) , u_{k} \rangle }  \leq \frac{C}{\Lambda_{k+1}}.
\end{equation}
By Proposition \ref{Propabc}(e),
in order to prove (\ref{ceq}), it is enough to show that there exists an $L > 0$ such that
\begin{equation}\label{eqx}
\frac{1}{L} \leq \frac{\langle (\mathbf{1} - f_{k+1,n}(s))^THg_{k}(\eta_{k+1,n})(\mathbf{1} - f_{k+1,n}(s)) , u_{k} \rangle}{\langle \mathbf{1} - f_{k+1,n}(s) ,u_{k+1} \rangle \langle \mathbf{1} - f_{k,n}(s) , u_{k} \rangle } \leq L.
\end{equation}

Now, we know that  $f^j_{k,n}(\mathbf{0}) \leq f^j_{k,n}(s) \leq \eta_{k+1,n} (j) \leq 1$ for each $k$ and $ j \in S$. Also, $f^j_{k,n}(\mathbf{0})  = \mathbb{P}(Z_n = \textbf{0}| Z_k = e_j) \geq \varepsilon_0$
for each $k \leq n-1$, and thus $\varepsilon_0  \leq \eta_{k+1,n} (j) \leq 1$ for each $k \leq n-1$ and $ j \in S$.
Thus, by Assumptions 1-3, there exists a constant $c_1 > 0$ such that for each vector $\zeta$ with non-negative components, we have
\[
\frac{1}{c_1}\|\zeta\|^2 \leq \langle \zeta^THg_{k}(\eta_{k+1,n})\zeta , u_{k} \rangle \leq c_1\|\zeta\|^2.
\]
In particular, we have 
\begin{equation} \label{eqy}
\frac{1}{c_1}\|\mathbf{1} - f_{k+1,n}(s)\|^2 \leq \langle (\mathbf{1} - f_{k+1,n}(s))^THg_{k}(\eta_{k+1,n})(\mathbf{1} - f_{k+1,n}(s)) , u_{k} \rangle 
\leq c_1\|\mathbf{1} - f_{k+1,n}(s)\|^2.
\end{equation}

By Proposition~\ref{Propabc}, for each $ 0 \leq k \leq n- 1$, 
\[
\bar{\varepsilon}\|\mathbf{1} - f_{k,n}(s)\| \leq \langle \mathbf{1} - f_{k,n}(s) , u_{k} \rangle \leq \|\mathbf{1} - f_{k,n}(s)\|.
\]
In order to prove (\ref{eqx}), it is sufficient to prove that there exists a constant $c_2 >0$ such that for each $0 \leq k \leq n-1$ and each $s \in [0,1]^d 
\setminus \{ \mathbf{1} \}$,
\[
\frac{1}{c_2} \leq \frac{\|\mathbf{1} - f_{k+1,n}(s)\|}{\|\mathbf{1} - f_{k,n}(s)\|}   \leq c_2.
\]
The first inequality $\|\mathbf{1} - f_{k,n}(s)\| \leq c_2 \|\mathbf{1} - f_{k+1,n}(s)\|$ follows from the fact that  
\[
\|\mathbf{1} - f_{k,n}(s)\| = \|g_{k}(\mathbf{1}) - g_k(f_{k+1,n}(s))\| \leq c_2 \|\mathbf{1} - f_{k+1,n}(s)\|,
\]
since $g_k$ is uniformly Lipschitz due to Assumption 3.

We observe that by Assumptions 1-3, each entry of the matrix $A_k$ is uniformly bounded from above and below, i.e., there exist positive constants $r$ and $R$ such that, for each $i,j \in S$,
\[
r \leq A_k(i,j) \leq R.
\]
To prove the second inequality $\|\mathbf{1} - f_{k+1,n}(s)\| \leq c_2 \|\mathbf{1} - f_{k,n}(s)\|$, we consider the following two cases:\\

(CASE I) $\|\mathbf{1} - f_{k+1,n}(s)\|  \leq r\bar{\varepsilon}d/c_1$. Then, from equation (\ref{eqy}) and Proposition \ref{Propabc}, 
\[
\langle (\mathbf{1} - f_{k+1,n}(s))^THg_{k}(\eta_{k+1,n})(\mathbf{1} - f_{k+1,n}(s)), u_k \rangle  \leq 
c_1\|\mathbf{1} - f_{k+1,n}(s)\|^2.
\]
\[
\leq  r\bar{\varepsilon}d \|\mathbf{1} - f_{k+1,n}(s)\| \leq \langle A_k (\mathbf{1} - f_{k+1,n}(s)), u_k \rangle,  
\]
and thus, substituting the above relation into the Taylor formula,
\[
\langle \mathbf{1} - f_{k,n}(s) , u_{k} \rangle  = \langle A_{k}
(\mathbf{1} - f_{k +1,n}(s)) , u_{k} \rangle - \frac{1}{2}\langle(\mathbf{1} - 
f_{k+1,n}(s))^TH  g_{k} (\eta_{k + 1,n})(\mathbf{1} - f_{k +1,n}(s)) , u_{k} \rangle,
\]
we get, 
\[
\langle \mathbf{1} - f_{k,n}(s) , u_{k} \rangle  \geq \langle A_{k}(\mathbf{1} - f_{k +1,n}(s)) , u_{k} \rangle /2, 
\]
and thus, 
\[
\|\mathbf{1} - f_{k,n}(s)\| \geq \langle \mathbf{1} - f_{k,n}(s) , u_{k} \rangle \geq  \langle A_{k}(\mathbf{1} - f_{k +1,n}(s)) , u_{k} \rangle /2 \geq 
r \bar{\varepsilon}\|\mathbf{1} - f_{k+1,n}(s)\|/2.
\] 
So, for $\tilde{c_2} = 2/( r \bar{\varepsilon}) $, we have 
\[
\|\mathbf{1} - f_{k+1,n}(s)\| \leq \tilde{c_2} \|\mathbf{1} - f_{k,n}(s)\|. 
\]

(CASE II) Now suppose that $ {1}  - f^j_{k+1,n}(s) > r\bar{\varepsilon}/c_1$ for some $j \in S$. 
We want to prove that there exists a $\gamma >0$ such that $ {1}  - f^j_{k,n}(s) \geq \gamma$. 
From Assumptions 1-2, for each $j  \in S$,
\[
g_k^j(s) = \mathbb{E}\Big(\prod_{i = 1}^{d} s_i^{Z_{k+1}(i)} | Z_k = e_j\Big) \leq (1 - \varepsilon_0) + \varepsilon_0  s_j^2, 
\]
and thus, since $f^j_{k+1,n}(s) < 1 - r\bar{\varepsilon}/c_1$, 

\[
f_{k,n}^j(s) = g_k^j(f_{k+1,n}(s)) \leq (1 - \varepsilon_0) + \varepsilon_0 (1 - r\bar{\varepsilon}/c_1)^2 < 1, 
\]
where the last inequality holds since $0 < r\bar{\varepsilon}/c_1 < 1$. 
Setting $\gamma =  \varepsilon_0 - \varepsilon_0 (1 - r\bar{\varepsilon}/c_1)^2$, we obtain
\[
{1}  - f^j_{k,n}(s) > \gamma,  
\]
which is the required inequality.

So, from the two cases above, we can define $c_2 = \max(\tilde{c_2},  d/\gamma )$ to get,  for each $0 \leq k \leq n-1$ 
and $s \in [0,1]^d  \setminus \{ \mathbf{1}\} $,
\[
 \|\mathbf{1} - f_{k+1,n}(s)\| \leq c_2 \|\mathbf{1} - f_{k,n}(s)\|. 
\]
\end{proof}


\begin{proof}
[Proof of Lemma~\ref{elesst}] Let $\bar{u} = \mathbb{E} Z_1$. By Assumptions 1-3, for every initial population, there is a constant $c > 0$ such that $c^{-1} u_1 \leq \bar{u} \leq c u_1$, where inequality between vectors is understood as the inequality 
between their components. Then, since $ M_{1,n}^T \bar{u} =  \mathbb{E} Z_n$,
\[
c^{-1} M_{1,n}^T u_1 \leq  \mathbb{E} Z_n \leq c M_{1,n}^T u_1.
\]
 Taking the norm and using the fact that $M_{1,n}^T u_1 = (\Lambda_n/\Lambda_1) u_n$, we obtain (\ref{qwq1}). 

From (\ref{hhk}) with  $s = \mathbf{0}$, using the fact that $1 - f_{0,n}^j(\mathbf{0}) = \mathbb{P}(Z_n \neq \mathbf{0} | Z_0 = e_j)$, we obtain
\[
\frac{1}{C}  \Big(\frac{1}{\tilde{\Lambda}_{n}} + \alpha(n,s)\Big)^{-1} \leq \mathbb{P}(_jZ_n \neq \mathbf{0} )  \leq C  \Big(\frac{1}{\tilde{\Lambda}_{n}} + \alpha(n,s)\Big)^{-1}.
\]
Using Lemma~\ref{lemma1} and the first estimate in part (e) of Proposition~\ref{Propabc}, we obtain, for a different constant $C$,
\[
\frac{1}{C}  \Big(\frac{1}{{\Lambda}_{n}} + \Xi_n \Big)^{-1} \leq \mathbb{P}(_jZ_n \neq \mathbf{0} )  \leq C  \Big(\frac{1}{{\Lambda}_{n}} +\Xi_n \Big)^{-1}.
\]
Since this is valid for every $j$, we have the same inequality for an arbitrary initial population (with a constant $C$ that depends on the initial population). 
Since $\Lambda_n \Xi_n \geq 1$, this implies~(\ref{qwq2}).  
\end{proof}

\section{Convergence of the process conditioned on survival} \label{sec3}

The following series will be important to our analysis,
\begin{equation} \label{degam} 
\Gamma_n = \frac{1}{2} \sum_{k = 0}^{n-1}\frac{1}{\lambda_k\tilde{\Lambda}_{k+1}}\frac{\langle v_{k+1}^THg_{k}(\mathbf{1})v_{k+1} , u_{k} \rangle}{\langle v_{k+1} ,u_{k+1} \rangle \langle v_{k}, u_{k} \rangle}.
\end{equation}
Here $H$ denotes the Hessian matrix. It is applied to each component of $g_k$ separately, then multiplied by vectors $v_{k+1}^T$ and $v_{k+1}$ to
get scalars, which are then multiplied by the corresponding components of $u_k$ to form the scalar product in the numerator. Since all terms in the right side of (\ref{degam}) are positive, the sequence $\Gamma_n$ is increasing.  In each term in~(\ref{degam}), each of the 
factors, $\lambda_k$, $\langle v_{k+1}^THg_{k}(\mathbf{1})v_{k+1} , u_{k} \rangle$, $\langle v_{k+1} ,u_{k+1} \rangle$, and $\langle v_{k}, u_{k} \rangle$, is bounded 
from above and below uniformly in $k$ by Assumptions 1-3 and Proposition~\ref{Propabc}. Therefore, 
by Proposition~\ref{Propabc}, there is a positive constant $C$ such that
\begin{equation} \label{esnnze}
\frac{1}{C \Lambda_{k+1}} \leq  \frac{1}{2} \frac{1}{\lambda_k\tilde{\Lambda}_{k+1}}\frac{\langle v_{k+1}^THg_{k}(\mathbf{1})v_{k+1} , u_{k} \rangle}{\langle v_{k+1} ,u_{k+1} \rangle \langle v_{k}, u_{k} \rangle} \leq \frac{C}{\Lambda_{k+1}}~,
\end{equation}
and consequently,
\begin{equation} \label{esnn}
\frac{\Xi_n}{C} = \frac{1}{C} \sum_{k = 0}^{n-1}\frac{1}{\Lambda_{k+1}} \leq \Gamma_n   \leq  C \sum_{k = 0}^{n-1}\frac{1}{\Lambda_{k+1}}  =
C \Xi_n.
\end{equation}
Assumptions 5 and 6 now can be rewritten as 
\begin{equation} \label{aeit}
\Gamma_n \rightarrow \infty,~~\Lambda_n \Gamma_n \rightarrow \infty~~{\rm as}~n \rightarrow \infty. 
\end{equation}

The proof of Theorem~\ref{the2}  will rely on the following seemingly weaker statement. 
\begin{thm} \label{nnn} Under Assumptions 1-6, for each $j \in S$, we have the following limit in distribution
 \[
\frac{ \langle _j\zeta_n, u_n\rangle}{  \mathbb{E} \langle _j\zeta_n, u_n\rangle} \rightarrow \xi ~~{\it as}~n \rightarrow \infty,
 \]
where $\xi$ is an exponential random variable with parameter one.
\end{thm}
\begin{proof} The proof will rely on several lemmas, which will be formulated as needed. The proofs of these lemmas will be given in the end of this section.  It is sufficient to show convergence of moment generating functions. That is, we want to prove that for each $\varkappa \in  [0,\infty)$
\[
 \mathbb{E} \Big(\exp\Big({\frac{-\varkappa\langle _jZ_n, u_n\rangle \mathbb{P}(_jZ_n \neq \mathbf{0})}{\mathbb{E}(\langle _jZ_n , u_n\rangle)}}\Big) \Big|\; _jZ_n \neq \mathbf{0}\Big) \rightarrow \frac{1}{1 + \varkappa}~~~{\rm as}~n \rightarrow \infty.
\]
Let us define vectors $\bar{s}_j $ such that the $i$-th component of $\bar{s}_j $ is 
\[
\bar{s}_j (i) = \exp\left({\frac{-\varkappa u_n(i)\mathbb{P}(_jZ_n \neq \mathbf{0})}{\mathbb{E}(\langle _jZ_n , u_n\rangle)}}\right).
\]
Then the $j$-th component of the vector $f_n(\bar{s}_j )$ is equal to
\[
f_n^j(\bar{s}_j ) = \mathbb{E} \Big(\exp\Big({\frac{-\varkappa \langle _jZ_n, u_n\rangle 
\mathbb{P}(_jZ_n \neq \mathbf{0})}{\mathbb{E}(\langle _jZ_n , u_n\rangle)}}\Big)\Big). 
\] 
Thus we want to show that
\begin{equation} \label{ew1}
  1 - \frac{1 - f_n^j(\bar{s}_j )}{\mathbb{P}(_jZ_n \neq \mathbf{0})} \rightarrow \frac{1}{1 + \varkappa }~~{\rm as}~n \rightarrow \infty. 
\end{equation}

In order to prove (\ref{ew1}), it will be useful to study the asymptotic behavior of the sum on the right hand side of (\ref{eq1z}). We first find the
upper and lower bounds of the sum using the upper and lower bounds for $\eta_{k,n}$.
Observe that $Hg_{k}^j(s)$ is monotonic in $s$ for each $j$ since $g_k^j$ is a polynomial  with non-negative coefficients and $Hg_{k}^j(s)$ is a matrix with entries that are mixed second derivatives of $g_k^j$. Therefore,
\eqref{eq1z} gives
\[
  \Big(\frac{1}{\tilde{\Lambda}_{n}\langle \mathbf{1}-s , u_n\rangle} + \frac{1}{2} \sum_{k = 0}^{n-1} \frac{\langle (\mathbf{1} - f_{k+1,n}(s))^THg_{k}(\mathbf{1})(\mathbf{1} - f_{k+1,n}(s)) , u_{k} \rangle}{\tilde{\Lambda}_{k+1}\langle (\mathbf{1} - f_{k+1,n}(s)) ,u_{k+1} \rangle \langle \mathbf{1} - f_{k,n}(s) , u_{k} \rangle }\Big)^{-1}
\]
\begin{equation} \label{bigsum}
\leq \langle \mathbf{1} - f_{0,n}(s) , u_{0} \rangle 
\end{equation}
\[
\leq \Big(\frac{1}{\tilde{\Lambda}_{n}\langle \mathbf{1}-s , u_n\rangle} + \frac{1}{2}\sum_{k = 0}^{ n-1} \frac{\langle (\mathbf{1} - f_{k+1,n}(s))^THg_{k}(f_{k,n}(s))(\mathbf{1} - f_{k+1,n}(s)) , u_{k} \rangle}{\tilde{\Lambda}_{k+1}\langle (\mathbf{1} - f_{k+1,n}(s)) ,u_{k+1} \rangle \langle \mathbf{1} - f_{k,n}(s) , u_{k} \rangle }\Big)^{-1}.
\]
Let us briefly explain the idea for the next step. Assume that $K$ is such that $n -K$ is large and $f_{k,n}(s)$ is close to $\mathbf{1}$ for $k \leq K+1$. By formally linearizing the mappings $g_k$, $g_{k+1}$,...,$g_K$, we write
\begin{equation}
\label{FKsteps}
 \mathbf{1} - f_{k,n}(s) \approx A_{k}A_{k+1}\cdots A_{K}(\mathbf{1} - f_{K+1,n}(s)).
\end{equation}
We know that
\[
 v_k =   A_{k}\frac{v_{k+1}}{\lambda_{k}},
\]
and thus
\begin{equation}
\label{VKsteps}
 v_k =  A_{k} A_{k+1}\cdots A_{K}\frac{v_{K+1}}{\prod_{i = k}^{K}\lambda_{i}}.
\end{equation}
Note the similarity in the expressions 
\eqref{FKsteps} and \eqref{VKsteps}:
the same product of matrices is applied, albeit to different vectors.
 Proposition $\ref{Propabc}$$(d)$ (contractive property of the matrices)  implies that the resulting expressions will be aligned in the same direction if $K - k$ is sufficiently large. That is we can replace
$\mathbf{1} - f_{k,n}(s)$ (and $\mathbf{1} - f_{k+1,n}(s)$) by the vectors $c_{k,n} v_k$ (and $c_{k+1,n} v_{k+1}$) in each of the terms in the sums in (\ref{bigsum}) for all $k$ that are sufficiently far away from $n$, where $c_{k,n}$ satisfy the relation ${c_{k,n}}/{c_{k+1,n}} = \lambda_k$. This will simplify (\ref{bigsum}). 

Now let us make the above arguments rigorous.
For a given $\varepsilon > 0$ and a positive integer $n$, we define $J(n,\varepsilon)$ as follows,
\[
J(n,\varepsilon) = \min\{k: 1 - f^i_{k,n}(\textbf{0}) > \varepsilon \text{ for some } i \in S\}.
\]
\begin{lem} \label{lemmaa}
 For each $\varepsilon' > 0 $, there exist a natural number $K$ and an $\varepsilon > 0 $ such that, 
for each $s \in [0,1]^d  \setminus \{ \mathbf{1} \} $, 
\begin{equation} \label{xzx}
\mathbf{1} - f_{k,n}(s) = c_{k,n}(v_{k} + \delta_{k,n}),
\end{equation}
 where $\delta_{k,n}$ and $c_{k,n}$ depend on $s$ and satisfy  $\|\delta_{k,n}\| \leq \varepsilon'$  and $|({c_{k,n}}/{c_{k+1,n}}) - \lambda_k| \leq \varepsilon'$ for each $0 \leq k \leq J(n,\varepsilon) - K$ and each $n$.
\end{lem}

\
\\
Note that $J(n, \varepsilon) \rightarrow \infty$ as $n \rightarrow \infty$  since each component of the vector
$\mathbf{1} - f_{k,n}(\textbf{0})$ is 
\[
1 - f_{k,n}^i(\textbf{0}) = \mathbb{P}(Z_n \neq \textbf{0}| Z_k = e_i)
\]
and  $\mathbb{P}(Z_n \neq \textbf{0}| Z_k = e_i) \rightarrow 0$ as  $n  \rightarrow \infty$ for each $i$ and each $k$
by Lemma~\ref{ljj3}. 
\\

Recall the definition of $\alpha(n,s)$ from (\ref{lll}).
\begin{lem} \label{lemmab} Under Assumptions 1-6,  
\[
\lim_{n \rightarrow \infty} \frac{\alpha(n,s)}{\Gamma_n} = 1,
\]
uniformly in $s \in [0,1]^d  \setminus \{ \mathbf{1} \} $. 
\end{lem}
\
\\
Let us return to the proof of (\ref{ew1}). 
By Lemma \ref{lemmaa}, when $n$ is large, the vector $\mathbf{1}-f_{n}(s) = \mathbf{1}-f_{0,n}(s)$ is nearly aligned to the vector $v_0$. Thus, in (\ref{ew1}) we can replace the $j$-th component of the vector $\mathbf{1}-f_{n}(\bar{s}_j )$ by
 \[
  \langle \mathbf{1} - f_n(\bar{s}_j )  , u_0 \rangle \frac{\langle v_0, e_j\rangle}{\langle v_0 , u_0 \rangle}.
 \]
Therefore, in order to prove (\ref{ew1}), it is sufficient to show that 
 \[
  1 - \frac{\langle v_0, e_j\rangle}{\langle v_0 , u_0 \rangle\mathbb{P}(_jZ_n \neq \mathbf{0})}
  \left(\frac{1}{\tilde{\Lambda}_{n}\varkappa \Big(\frac{\mathbb{P}(_jZ_n \neq \mathbf{0})}{\mathbb{E}(\langle _jZ_n , u_n\rangle)} 
  +o\left(\frac{\mathbb{P}(_jZ_n \neq \mathbf{0})}{\mathbb{E}(\langle _jZ_n , u_n\rangle)}\right)\Big)\langle u_n , u_n\rangle} + \Gamma_n\right)^{-1} \rightarrow \frac{1}{1+ \varkappa },
 \]
where we used Lemma~\ref{lemmab} to transform (\ref{hhk}) and linearized $\textbf{1} - \bar{s}_j $. 
The LHS can be written as
\begin{equation} \label{ey}
   1 - \frac{\langle v_0, e_j\rangle}{\langle v_0 , u_0 \rangle\mathbb{P}(_jZ_n \neq \mathbf{0})\Gamma_n}
   \left(\frac{1}{\tilde{\Lambda}_{n}\Gamma_n\varkappa \Big(\frac{\mathbb{P}(_jZ_n \neq \mathbf{0})}{\mathbb{E}(\langle _jZ_n , u_n\rangle)} +o\left(\frac{\mathbb{P}(_jZ_n \neq \mathbf{0})}{\mathbb{E}(\langle _jZ_n , u_n\rangle)}\right)\Big)
   \langle u_n , u_n\rangle} + 1 \right)^{-1}.
\end{equation}
We will need  the following two lemmas.
 \begin{lem} \label{lemmac} Under Assumptions 1-6,
 \[
  \lim_{n \rightarrow \infty}   \frac{\langle v_0 , u_0 \rangle\mathbb{P}(_jZ_n \neq \mathbf{0})\Gamma_n}{\langle v_0, e_j\rangle} = 1.
 \]
 \end{lem}

\begin{lem} \label{lemmaz}Under Assumptions 1-6, 
\[
 \lim_{n \rightarrow \infty} \tilde{\Lambda}_{n}\Gamma_n \frac{\mathbb{P}(_jZ_n \neq \mathbf{0})  \langle u_n , u_n\rangle}{\mathbb{E}(\langle _jZ_n , u_n\rangle) } = 1.
\]
\end{lem}
\
\\
Applying the above two lemmas to transform the expression in (\ref{ey}), we obtain
\[
\lim_{n \rightarrow \infty} \left(1 - \frac{\langle v_0, e_j\rangle}{\langle v_0 , u_0 \rangle\mathbb{P}(_jZ_n \neq \mathbf{0})\Gamma_n}\left(\frac{1}{\tilde{\Lambda}_{n}\Gamma_n\varkappa \Big(\frac{\mathbb{P}(_jZ_n \neq \mathbf{0})}{\mathbb{E}(\langle _jZ_n , u_n\rangle)} +
o\left(\frac{\mathbb{P}(_jZ_n \neq \mathbf{0})}{\mathbb{E}(\langle _jZ_n , u_n\rangle)}\right)\Big)\langle u_n , u_n\rangle} + 1 \right)^{-1}\; \right) 
\]
\[
=  1 - \Big(\frac{1}{\varkappa } + 1\Big)^{-1} = \frac{1}{1 + \varkappa } .
\]
This completes the proof of Theorem~\ref{nnn}. 
\end{proof}

\begin{proof}[Proof of Theorem~\ref{the2}]  
First, let Assumptions 1--6 be satisfied. Let $\cP: v \rightarrow v/\|v\|$ be the projection onto the
unit sphere, with the convention that $\cP(\mathbf{0}) = \mathbf{0}$. 
We claim that
\begin{equation} \label{colli}
\lim_{n \rightarrow \infty} \| \cP (\mathbb{E} _j \zeta_n) - u_n \| = 0,
\end{equation}
\begin{equation} \label{colli2}
\lim_{n \rightarrow \infty} \mathbb{P} (\| \cP \left( _j \zeta_n \right)- u_n \| > \varepsilon) = 0
\end{equation}{sec3}
for each $j$ and  each $\varepsilon > 0$. Let us fix $\delta \in (0, \varepsilon)$.  By Proposition~\ref{Propabc}, we can find $k' \in \mathbb{N}$ such that
\begin{equation} \label{cojjj}
(1-{\delta})u_{n+k'} \leq \frac{M^T_{n, n+k'} u}{\|M^T_{n, n+k'} u\|} \leq (1+{\delta})u_{n+k'}
\end{equation}
whenever $u$ is a non-zero vector with non-negative components.  Let $_j\zeta^{k'}_n$ be the random vector obtained by taking 
$_j\zeta_n$ as the initial population of a branching process, then branching for $k'$ steps using 
 our original branching distributions $P_n,...,P_{n+k'-1}$
and evaluating the
resulting population. Note that $_j\zeta^{k'}_n$ is different from $_j\zeta_{n+k'}$, the latter can be obtained from $_j\zeta^{k'}_n$ by conditioning on the event
of non-extinction. Since the extinction of a large initial population in $k'$ steps occurs with small probability and since, by Theorem~\ref{nnn}, for each $a > 0$ we have  $\mathbb{P}(\|_j\zeta_n \| > a ) \rightarrow 1$ as $n \rightarrow \infty$, we obtain
\[
\lim_{n \rightarrow \infty} ( \mathbb{P} (\| \cP ( _j \zeta^{k'}_n)- u_{n+k'} \| > \varepsilon) -  
\mathbb{P} (\| \cP ( _j \zeta_{n+k'} )- u_{n+k'} \| > \varepsilon)) = 0.
\]
Also note that $\lim_{n \rightarrow \infty}(  \cP (\mathbb{E} _j \zeta^{k'}_n) - 
\cP (\mathbb{E} _j \zeta_{n+k'}) ) = 0$.  
Therefore, since $\delta > 0$ was arbitrarily small, (\ref{colli}) and  (\ref{colli2}) will follow if we show that
\begin{equation}
\label{ProjExp}
 \| \cP (\mathbb{E} _j \zeta^{k'}_n) - u_{n+k'} \| \leq \delta
\end{equation}
for all sufficiently large $n$ and
\begin{equation}
\label{ProjProb}
 \lim_{n \rightarrow \infty} \mathbb{P} (\| \cP ( _j \zeta^{k'}_n )- u_{n+k'} \| > \varepsilon) = 0.
\end{equation}
\eqref{ProjExp} immediately follows from (\ref{cojjj}). \eqref{ProjProb} is a consequence of
\[
 \lim_{n \rightarrow \infty} \mathbb{P} (\| \cP ( _j \zeta^{k'}_n) - \cP (\mathbb{E} _j \zeta^{k'}_n )\| > 
 \varepsilon - \delta ) = 0,
\]
which can be derived from the Chebyshev inequality since for each $a > 0$ we have  
$\mathbb{P}(\|_j\zeta_n \| > a ) \rightarrow 1$ as $n \rightarrow \infty$. Thus we have (\ref{colli}) and  (\ref{colli2}). 

Next, let us show that $(\ref{colli})$ and  (\ref{colli2}), along with Theorem $\ref{nnn}$, imply $(\ref{maineq})$ with $_j\zeta_n$ in place of $\zeta_n $. By  Theorem $\ref{nnn}$, it is sufficient to show that we have the following limit in probability
\begin{equation}\label{qq}
\frac{ \langle _j\zeta_n, u_n\rangle}{  \mathbb{E} \langle _j\zeta_n, u_n\rangle} - \frac{ \langle _j\zeta_n, u\rangle}{  \mathbb{E} \langle _j\zeta_n, u\rangle} \rightarrow 0 ~~~ \text{as} ~~~ {n \to \infty}.
\end{equation}

From $(\ref{colli})$ we know that 
$$
\lim_{n \to \infty} (  \langle\mathcal{P}(\mathbb{E} (_j\zeta_n)), u_n\rangle- \langle u_n , u_n \rangle  ) = 0
\quad\text{and}\quad
\lim_{n \to \infty} ( \langle\mathcal{P}(\mathbb{E} (_j\zeta_n)), u\rangle - \langle u_n , u \rangle) = 0.
$$
By \eqref{colli2} the two following limits hold in probability,
\[
\lim_{n \to \infty} ( \langle\mathcal{P}(_j\zeta_n), u_n\rangle - \langle u_n , u_n \rangle) = 0
\quad\text{and}\quad
\lim_{n \to \infty}  (\langle\mathcal{P}(_j\zeta_n), u\rangle - \langle u_n , u \rangle ) = 0.
\]
Therefore, the right hand side in the following equality
\[
\frac{ \langle _j\zeta_n, u_n\rangle}{  \mathbb{E} \langle _j\zeta_n, u_n\rangle} - \frac{ \langle _j\zeta_n, u\rangle}{  \mathbb{E} \langle _j\zeta_n, u\rangle} 
=\frac{||_j\zeta_n||}{||\EXP(_j\zeta_n)||}\Big( \frac{ \langle \mathcal{P} (_j\zeta_n), u_n\rangle}{   \langle\mathcal{P}(\mathbb{E} (_j\zeta_n)), u_n\rangle} - \frac{ \langle \mathcal{P} (_j\zeta_n), u\rangle}{    \langle\mathcal{P}(\mathbb{E} (_j\zeta_n)), u\rangle} \Big)
\]
tends to zero in probability (the factor in the brackets tends to zero in probability while the first factor is bounded in $L^1$). This justifies $(\ref{qq})$ and therefore  $(\ref{maineq})$ with $_j\zeta_n$ in place of $\zeta_n $.

If the initial population of the process $\{Z_n\}$ 
is such that we have more than one particle at time zero,
then we can consider a new process $\{Z'_n\}$ 
for which  $Z_0'=e_1$ and 
the transition distribution $P'_0$ is such that $_j Z'_1$ 
coincides in distribution
with $Z_1$. We also define $P'_n = P_n$ for $n \geq 1$. 
It is easy to see that the modified process satisfies Assumptions 1-4 
with possibly different values of $\eps_0$ and $K_0.$ On the other hand, 
${\langle \zeta_n,u \rangle}/{  \mathbb{E} \langle \zeta_n,u \rangle } $ 
is equal, in distribution,
to ${\langle _j\zeta'_n,u \rangle}/{  \mathbb{E} \langle _j\zeta'_n,u \rangle } $
 when $n \geq 1$, and therefore  (\ref{maineq}) holds for every initial population. 

Finally, suppose that Assumptions 1-5 are satisfied. If Assumption 6 fails, then  
$\mathbb{E} \|\zeta_n \| = \mathbb{E} \| Z_n \| /\mathbb{P} (Z_n \neq \mathbf{0})$ is bounded along a subsequence for every initial population. Then (\ref{maineq}) does not hold since $\zeta_n$ is
integer-valued, which gives a contradiction. 
\end{proof}

\begin{proof}
[Proof of Lemma \ref{CrJagers}]
From Proposition \ref{PrRarify} it follows that if $\{Z_n\}$ satisfies the assumptions of 
Lemma \ref{CrJagers}, 
 then there exists $l$ such that $\{\tilde{Z}_n\}=\{Z_{nl}\}$ satisfies assumptions 1-3. By Theorem \ref{yth}, which can be applied due to $(\ref{UniCrit})$, 
$\{\tilde{Z_n}\}$ goes extinct almost surely. This implies the almost sure extinction of $\{Z_n\}$.

Similarly, from Theorem \ref{the2}  (which can be applied since Assumptions 5 and 6 are met by the process $\{Z_{nl}\}$ due to $(\ref{UniCrit})$), it follows that \eqref{maineq} holds along a subsequence $nl$. 
To show that \eqref{maineq} holds (without restriction to a subsequence) let $n=Nl+r$
with $0\leq r<l.$
We claim that for every $u$ and every $0 \leq r < l$, 
\begin{equation}
\label{RatioCSurv}
\lim_{N \rightarrow \infty} (\frac{\langle \zeta_{Nl+r}, u \rangle}{\langle \zeta_{Nl}, u_{Nl} \rangle}-
\frac{\Lambda_{Nl+r}}{\Lambda_{Nl}} \langle u_{Nl+r}, u\rangle ) = 0
\quad\text{in probability}.
\end{equation}
Lemma \ref{CrJagers} follows directly from \eqref{RatioCSurv}
and Theorem \ref{the2} applied to $\{Z_{Nl}\}.$ 
The proof of \eqref{RatioCSurv} is similar to the proof of \eqref{colli2} and \eqref{qq}, so we leave
it to the reader.
\end{proof}

It still remains to prove Lemmas \ref{lemmaa}--\ref{lemmaz}.

\begin{proof}[Proof of Lemma \ref{lemmaa}]   
Suppose that we have (\ref{xzx}) with $\|\delta_{k,n}\| \leq \varepsilon''$, but without any assumptions on $c_{k,n}$. Then we have, for  $0 \leq k < J(n,\varepsilon) - K$,
\[
\mathbf{1} - f_{k,n}(s) = \mathbf{1} - g_k(f_{k+1,n}(s)) =  A_k(\mathbf{1} - f_{k+1,n}(s)) + \alpha_{k,n} \|\mathbf{1} - f_{k+1,n}(s) \|,
\]
where $\|\alpha_{k,n}\|$ can be made arbitrarily small, uniformly in $k$, by selecting sufficiently small $\varepsilon$. The latter statement about $\alpha_{k,n}$ follows
from the assumption that $1 - f^i_{k+1,n}(s) \leq \varepsilon$ for all $i$ (from definition of $J(n,\varepsilon)$) and the fact that
\[
A_{k}(j,i) = \frac{\partial g_{k}^j}{\partial s_i}(\mathbf{1}).
\]
The uniformity in $k$ follows from Assumption 3. Thus 
\[
\mathbf{1} - f_{k,n}(s) =  A_k (c_{k+1,n}(v_{k+1} + \delta_{k+1,n})) + \alpha_{k,n} \|\mathbf{1} - f_{k+1,n}(s) \| = 
\]
\[
c_{k+1,n} \lambda_{k} v_k + c_{k+1,n}  A_k \delta_{k+1,n} + \alpha_{k,n} \|\mathbf{1} - f_{k+1,n}(s) \| = c_{k+1,n} \lambda_{k} (v_k + \alpha'_{k,n}),
\]
where $\|\alpha'_{k,n}\|$ can be made arbitrarily small, uniformly in $k$, by selecting sufficiently small $\varepsilon$ and $\varepsilon''$. Here we used (\ref{xzx})
with $k+1$ instead of $k$ to estimate the contribution from the term $\alpha_{k,n} \|\mathbf{1} - f_{k+1,n}(s) \|$. Thus
\[
c_{k+1,n} \lambda_{k} (v_k + \alpha'_{k,n}) = c_{k,n}(v_{k} + \delta_{k,n}),
\]
which implies that $|({c_{k,n}}/{c_{k+1,n}}) - \lambda_k| \leq \varepsilon'$ holds for $0 \leq k < J(n,\varepsilon) - K$, provided that $\varepsilon$ and $\varepsilon''$
are sufficiently small. We have demonstrated, therefore, that it is sufficient to establish (\ref{xzx}) with the estimate $\|\delta_{k,n}\| \leq \varepsilon'$ only.

By part (d) of Proposition~\ref{Propabc}, there is $k' \in \mathbb{N}$ such that
\begin{equation} \label{ooip}
\left(1-\frac{\varepsilon'}{2d}\right)v_k \leq \frac{M_{k, k+k'} v}{\|M_{k, k+k'} v\|} \leq 
\left(1+\frac{\varepsilon'}{2d}\right)v_k
\end{equation}
for each $k$ and each  non-zero vector $v$ with non-negative components. Since 
\[
M_{k, k+k'}(j,i) = \frac{\partial f_{k, k+k'}^j}{\partial s_i}(\mathbf{1}),
\]
we can linearize the mapping $\mathbf{1} - f_{k,k+k'}(s)$ at $s = \mathbf{1}$ and obtain that there is $\varepsilon$ such that
\[
{\|\mathbf{1} - f_{k,k+k'}(\mathbf{1} - v) - M_{k, k+k'}  v  \|} \leq \frac{\varepsilon'}{2d} {\|M_{k, k+k'} v \|}
\]
whenever $0 < \|v\| \leq \varepsilon d$. (We have used here that $M_{k, k+ k'}$ is bounded uniformly in $k$.) Therefore,
\[
M_{k, k+k'} v - \frac{\varepsilon'}{2d} \|M_{k, k+k'} v \| \mathbf{1} \leq \mathbf{1} - f_{k,k+k'}(\mathbf{1} - v) \leq M_{k, k+k'} v +\frac{\varepsilon'}{2d} \|M_{k, k+k'} v \| \mathbf{1}.
\]
Combined with (\ref{ooip}), this gives
\[
\|M_{k, k+k'} v \| (v_k - \frac{\varepsilon'}{d} \mathbf{1}) 
\leq \mathbf{1} - f_{k,k+k'}(\mathbf{1} - v) \leq \|M_{k, k+k'} v \| (v_k + \frac{\varepsilon'}{d} \mathbf{1}).
\]
Setting $K = k'+1$, we see that the last inequality can be applied to $v = \mathbf{1} - f_{k+k', n}(s)$, provided that $0 \leq k \leq J(n,\varepsilon) - K$, resulting in
\[
c_{k,n} (v_k - \frac{\varepsilon'}{d} \mathbf{1}) 
\leq \mathbf{1} - f_{k,n}(s) \leq c_{k,n} (v_k + \frac{\varepsilon'}{d} \mathbf{1}),
\]
which gives the desired estimate. 
\end{proof}

\begin{proof}[Proof of Lemma \ref{lemmab}] We break up the difference $ ({\alpha(n,s)}/{\Gamma_n}) -1$  into three parts. So we want to prove that for each $\sigma > 0$ there is $\varepsilon > 0$ such that   
\begin{equation} \label{threet}
\Big|\frac{\alpha(n,s) - {\alpha(J(n,\varepsilon) - K - 1}, s)}{\Gamma_n} + \frac{{\alpha(J(n,\varepsilon) - K - 1}, s) - \Gamma_{{J(n,\varepsilon) - K - 1}}}{\Gamma_n} 
\end{equation}
\[
+ \frac{\Gamma_{{J(n,\varepsilon) - K - 1}} - \Gamma_n}{\Gamma_n}\Big| \leq \sigma
\]
for all $s$ and all sufficiently large $n$.
\\

1. We first estimate the middle term in the inequality above. By Lemma \ref{lemmaa}, for each $\sigma' > 0$, there exist a natural number $K$ and $\varepsilon_1 > 0 $ such that
\[
(1 - \sigma')c_{k,n}v_k \leq  \mathbf{1} - f_{k,n}(s) \leq (1 + \sigma')c_{k,n}v_k
\]
for each $k < J(n,\varepsilon_1) - K$. 

By Assumption 4, $\|_iX_n\|^2$  are uniformly integrable, and thus the matrices $Hg_k^i(s)$, $k \geq 0$, $i \in S$, are equicontinuous in $s$. 
Note also that $\|Hg_k^i(\textbf{1})\| \geq c > 0$ for all $k \geq 0$, $i \in S$. Thus there exists $\varepsilon_2 > 0$ such that the matrix norm  
satisfies $\|Hg^i_{k}(\eta_{k+1,n}) - Hg^i_{k}(\textbf{1})\| < \sigma'\| Hg_k^i(\textbf{1})\| $ for each $k < J(n, \varepsilon_2) - K$.  Choosing  $\varepsilon = \min(\varepsilon_1, \varepsilon_2)$, we see that there is a constant $\tilde{c}$ independent of $\sigma' > 0$ such that
\[
\Big|\frac{1}{2}\sum_{k = 0}^{{J(n,\varepsilon) - K - 1}} \frac{\langle (\mathbf{1} - f_{k+1,n}(s))^THg_{k}(\eta_{k+1,n})(\mathbf{1} - f_{k+1,n}(s)) , u_{k} \rangle}{\tilde{\Lambda}_{k+1}\langle (\mathbf{1} - f_{k+1,n}(s)) ,u_{k+1} \rangle \langle \mathbf{1} - f_{k,n}(s) , u_{k} \rangle }  - \Gamma_{{J(n,\varepsilon) - K - 1}}\Big|
\]
\[
\leq \tilde{c} \sigma' \Gamma_{{J(n,\varepsilon) - K - 1}} \leq \tilde{c} \sigma' \Gamma_n.
\]
(In essence, there a small relative error, linear in $\sigma'$, in the factors in each of the terms of the sum, and thus the total relative error is small.)
By choosing $\sigma' \leq \sigma/(3 \tilde{c})$, we obtain
\[
\Big|\frac{{\alpha(J(n,\varepsilon) - K - 1}, s) - \Gamma_{{J(n,\varepsilon) - K - 1}}}{\Gamma_n}\Big| < \frac{\sigma}{3}.
\]

2. Now we estimate the third term in (\ref{threet}). We can assume that $K$ and $\varepsilon$ are fixed. We first observe that we can obtain a relation similar to 
(\ref{eq1z}) by starting with the expression $\langle \mathbf{1} - f_{{J(n,\varepsilon)},n}(s) , u_{{J(n,\varepsilon)}} \rangle $ instead of $\langle \mathbf{1} - f_{0,n}(s) , u_{0} \rangle$.
Thus, by doing the same steps that we carried out to obtain (\ref{eq1z}), we get
\[
\langle \mathbf{1} - f_{{J(n,\varepsilon)},n}(s) , u_{{J(n,\varepsilon)}} \rangle  
\]
\[
=  \Big(\frac{\tilde{\Lambda}_{J(n,\varepsilon)}}{\tilde{\Lambda}_{n}\langle \mathbf{1}-s , u_n\rangle} + \frac{1}{2} \sum_{k = {J(n,\varepsilon)}}^{n-1} \frac{\tilde{\Lambda}_{ J(n,\varepsilon) }\langle (\mathbf{1} - f_{k+1,n}(s))^THg_{k}(\eta_{k+1,n})(\mathbf{1} - f_{k+1,n}(s)) , u_{k} \rangle}{\tilde{\Lambda}_{k+1}
	\langle (\mathbf{1} - f_{k+1,n}(s)),u_{k+1} \rangle \langle \mathbf{1} - f_{k,n}(s) , u_{k} \rangle }\Big)^{-1}
\]
\[
= \frac{1}{\tilde{\Lambda}_{J(n,\varepsilon)}}\Big(\frac{1}{\tilde{\Lambda}_{n}\langle \mathbf{1}-s , u_n\rangle} + \frac{1}{2}\sum_{k = {J(n,\varepsilon)}}^{n-1} \frac{\langle (\mathbf{1} - f_{k+1,n}(s))^THg_{k}(\eta_{k+1,n})(\mathbf{1} - f_{k+1,n}(s)) , u_{k} \rangle}{\tilde{\Lambda}_{k+1}\langle (\mathbf{1} - f_{k+1,n}(s)) ,u_{k+1} \rangle \langle \mathbf{1} - f_{k,n}(s) , u_{k} \rangle }\Big)^{-1}
\]
\[
= \frac{1}{\tilde{\Lambda}_{J(n,\varepsilon)}}\Big(\frac{1}{\tilde{\Lambda}_{n}\langle \mathbf{1}-s , u_n\rangle} + (\alpha(n,s) - \alpha({J(n,\varepsilon)},s))\Big)^{-1}
\]
\[
\leq \frac{1}{\tilde{\Lambda}_{J(n,\varepsilon)} (\alpha(n,s) - \alpha({J(n,\varepsilon)},s))}
\leq C \frac{1}{\tilde{\Lambda}_{J(n,\varepsilon)} (\Gamma_n - \Gamma_{J(n,\varepsilon)})},
\]
where the last inequality follows from (\ref{ceq}) and (\ref{esnnze}). 
From here it follows that
\[
(\Gamma_n - \Gamma_{J(n, \varepsilon)}) \leq C \frac{1}{\tilde{\Lambda}_{J(n, \varepsilon)}\langle \mathbf{1} - f_{{J(n, \varepsilon)},n}(\textbf{0}) , u_{J(n, \varepsilon)} \rangle} \leq \frac{C}{\tilde{\Lambda}_{J(n, \varepsilon)}\varepsilon\bar{\varepsilon}}. 
\]
Therefore,
\[
\frac{\Gamma_n - \Gamma_{J(n, \varepsilon)}}{\Gamma_n} \leq \frac{C}{\tilde{\Lambda}_{J(n, \varepsilon)}\Gamma_n\varepsilon\bar{\varepsilon}} \leq \frac{C}{\tilde{\Lambda}_{J(n, \varepsilon)}\Gamma_{J(n, \varepsilon)}\varepsilon\bar{\varepsilon}}.
\]
Since $J(n, \varepsilon) \rightarrow \infty$  as $n \rightarrow \infty$, by Assumption 6 
(see (\ref{aeit}) and part (e) of Proposition~\ref{Propabc}),  we have 
\[
\left|\frac{C}{\tilde{\Lambda}_{J(n, \varepsilon)}\Gamma_{J(n, \varepsilon)}\varepsilon\bar{\varepsilon}}\right| < \frac{\sigma}{6}
\]
for all sufficiently large $n$. 
Now,
\[
\frac{\Gamma_n - \Gamma_{{J(n,\varepsilon) - K - 1}}}{\Gamma_n}  = \frac{\Gamma_n - \Gamma_{J(n, \varepsilon)}}{\Gamma_n} + \frac{\Gamma_{J(n, \varepsilon)} - \Gamma_{{J(n,\varepsilon) - K - 1}}}{\Gamma_n}. 
\]
For a fixed $K$, using Assumption 6 and the fact that $J(n,\varepsilon) \rightarrow \infty$ as $n \rightarrow \infty$, we see for all sufficiently large $n$ and
$k \geq J(n,\varepsilon) - K - 1$, 
\[
\left|\frac{1}{\Gamma_n}\left(\frac{1}{\lambda_k\tilde{\Lambda}_{k+1}}\frac{\langle v_{k+1}^THg_{k}(\mathbf{1})v_{k+1} , u_{k} \rangle}{\langle v_{k+1} ,u_{k+1} \rangle \langle v_{k}, u_{k} \rangle} \right)\right| < \frac{\sigma}{3(K + 1)}.
\]
Therefore, 
\[
\left|\frac{1}{\Gamma_n}\left(\frac{1}{2}\sum_{k = J(n, \varepsilon) - K - 1}^{J(n, \varepsilon)}\frac{1}{\lambda_k\tilde{\Lambda}_{k+1}}\frac{\langle v_{k+1}^THg_{k}(\mathbf{1})v_{k+1} , u_{k} \rangle}{\langle v_{k+1} ,u_{k+1} \rangle \langle v_{k}, u_{k} \rangle} \right)\right| < \frac{\sigma}{6}.
\]
Thus,
\[
\frac{\Gamma_{J(n, \varepsilon)} - \Gamma_{{J(n,\varepsilon) - K - 1}}}{\Gamma_n}  < \frac{\sigma}{6},
\]
and, therefore, for all sufficiently large $n$ we have
\[
\left|\frac{\Gamma_n - \Gamma_{J(n, \varepsilon)}}{\Gamma_n}\right| < \frac{\sigma}{3}.
\]

3. We know that 
\[
\frac{\alpha(n,s) - {\alpha(J(n,\varepsilon) - K - 1}, s)}{\Gamma_n} \leq C \frac{\Gamma_n - \Gamma_{{J(n,\varepsilon) - K - 1}}}{\Gamma_n},  
\]
and by the same arguments as above, for all sufficiently large $n$ we have
\[
\left|C \frac{\Gamma_n - \Gamma_{{J(n,\varepsilon) - K - 1}}}{\Gamma_n}  \right| <  \frac{\sigma}{3}.
\]

By the estimates from steps 1-3, for all sufficiently large $n$ and all ${s \in [0,1]^d  \setminus \{ \mathbf{1} \} }$ we have 
\[
\left|\frac{\alpha(n,s)}{\Gamma_n} - 1 \right| < \sigma.
\]
Since $\sigma > 0$ was arbitrary, the proof is complete. 
\end{proof}

\begin{proof}[Proof of Lemma \ref{lemmac}]  
We know that $\mathbb{P}(_jZ_n \neq \mathbf{0}) = 1 - f_n^j(\mathbf{0})$, and therefore for a fixed $j \in S$, by Lemma~\ref{lemmaa}, it is sufficient to prove that
\begin{equation} \label{sfor}
\lim_{n \rightarrow \infty} (c_{0,n} \langle v_0, u_0 \rangle \Gamma_n) = 1,
\end{equation}
where $c_{0,n}$ is the same as in Lemma~\ref{lemmaa}.
From Lemma~\ref{lemmab} and (\ref{hhk}) we know that
\[
\langle \mathbf{1} - f_{0,n}(s) , u_{0} \rangle \sim \Big(\frac{1}{\tilde{\Lambda}_{n}\langle \mathbf{1} -s , u_n\rangle} + \Gamma_n\Big)^{-1}~~~{\rm as}~n \rightarrow \infty. 
\]
By plugging in $s = \mathbf{0}$ and by replacing 
$\langle \mathbf{1} - f_{0,n}(\mathbf{0}) , u_{0} \rangle$ by $c_{0,n}\langle v_0, u_0\rangle$, we get
\[
\lim_{n \rightarrow \infty} c_{0,n}\langle v_0, u_0\rangle\Big(\frac{1}{\tilde{\Lambda}_{n}\langle \mathbf{1} , u_n\rangle} + \Gamma_n \Big) = 1.
\]
Thus we have
\[
\lim_{n \rightarrow \infty} c_{0,n}\Gamma_n\langle v_0, u_0\rangle
\Big(\frac{1}{\tilde{\Lambda}_{n}\Gamma_n  \langle \mathbf{1} , u_n\rangle} + 1 \Big) = 1.
\]
By Assumption 6, $\tilde{\Lambda}_{n}\Gamma_n \rightarrow \infty$ as $n\rightarrow \infty$  proving \eqref{sfor}.
\end{proof}

\begin{proof}[Proof of Lemma \ref{lemmaz}]
As in the proof of Lemma~\ref{lemmac}, it is sufficient to show that
\[
\lim_{n \rightarrow \infty}  \frac{c_{0,n}\tilde{\Lambda}_n\Gamma_n\langle v_0, e_j\rangle \langle u_n , u_n \rangle}{\mathbb{E}(\langle _jZ_n, u_n \rangle)} = 1.
\]
We observe that
\[
\mathbb{E}(\langle _jZ_n, u_n \rangle) = \sum_{i = 1}^d \mathbb{E}(Z_n(i)u_n(i) |Z_0 = e_j)
\]
\[
=\sum_{i = 1}^d u_n(i)\mathbb{E}(Z_n(i) |Z_0 = e_j)
\]
\[
=\sum_{i = 1}^d u_n(i)M_n(j,i)
\]
\[
= \langle M_n u_n, e_j\rangle
\]
\[
=  \langle  u_n, M_n^Te_j\rangle
\]
We know that $v_0 = M_nv_n/{\Lambda_n}$,
and thus what we want to prove is that
\[
\lim_{n \rightarrow \infty}  \frac{c_{0,n}\tilde{\Lambda}_n\Gamma_n\langle v_n, M_n^Te_j\rangle \langle u_n , u_n \rangle}{\Lambda_n\langle  u_n, M_n^Te_j\rangle} = 1.
\]
By (\ref{sfor}), it is sufficient to show that  
\[
\lim_{n \rightarrow \infty}
\frac{\tilde{\Lambda}_n\langle v_n, M_n^Te_j\rangle \langle u_n , u_n \rangle}{\Lambda_n\langle v_0, u_0\rangle\langle  u_n, M_n^Te_j\rangle} = 1.
\]
By Proposition \ref{Propabc} (part (d)),  the vectors  $M_n^Te_j$ 
align with the vectors $u_n$.
Therefore, it remains to prove that 
\[
\frac{\tilde{\Lambda}_n\langle v_n, u_n\rangle}{\Lambda_n\langle v_0, u_0\rangle} =  1.
\]
But this is true because
\[
\langle v_n , u_n\rangle  = \langle  v_n, \frac{ A^{T}_{n-1}u_{n-1}}{\tilde{\lambda}_{n-1}}\rangle
= \langle  A_{n-1}v_n , \frac{u_{n-1}}{\tilde{\lambda}_{n-1}}\rangle
= \frac{\lambda_{n-1}}{\tilde{\lambda}_{n-1}}\langle v_{n-1}, u_{n-1} \rangle
= \frac{\Lambda_{n}}{\tilde{\Lambda}_{n}}\langle v_{0}, u_{0} \rangle,
\]
where the last equality is obtained by iterating the previous steps $n$ times. 
\end{proof}

\section{Branching processes with continuous time} \label{brct}
In this section,  we provide an application of our results to continuous time branching.
Let $\rho_t(j)$, $1 \leq j \leq d$, be continuous functions  and $P_t(j,\cdot)$ be transition distributions on $\mathbb{Z}_+^d$ such that $P_t(j,a)$ is continuous
for each $a \in \mathbb{Z}_+^d$. 

Let $_jX_t$ be a random vector with values in $\mathbb{Z}_+^d$, whose distribution is given by $P_t(j,\cdot)$. We assume that there are $\varepsilon_0, K_0 > 0$ such that for all $i,j \in S$ the following bounds hold.
\\

0$'$. $\varepsilon_0 \leq \rho_t(j) \leq K_0$. 
\\

1$'$.  $\mathbb{P}(_jX_t(i) \geq 2) \geq \varepsilon_0$.
\\

2$'$. $\mathbb{P}(_jX_t = \mathbf{0} ) \geq \varepsilon_0$.
\\

3$'$. $\mathbb{E}(\|_jX_t\|^2 \big) \leq K_0$.
\\

Assuming that we start with a finite number of particles and that the above bounds hold, the transition rates $\rho_t(j)$ and the transition distributions $P_t(j,\cdot)$ define a  continuous time branching process $\{Z_t\}$ with particles of $d$ different types. Namely, each particle of 
type $j$ alive at time $t$ undergoes transformation into $a_1+...+a_d$ particles: $a_1$ particles of type one, $a_2$ particles of type two, etc., during the time interval $[t, t+ \Delta]$ with probability $ \rho_t(j) P_t(j,a) (\Delta + o(\Delta))$. 

Observe that $\{Z_n\}$, $n \in \mathbb{N}$, $n \geq 0$, is a discrete time branching process that satisfies Assumptions 1-3 (with different $\varepsilon_0$ and $K_0$). The fact that it satisfies Assumptions 1 and 2 is clear.
The first moment $M(t) = \mathbb{E} Z_t$ satisfies
\begin{equation} \label{hjhjl}
M'(t) = B^T(t) M(t),
\end{equation}
where $B(t)_{ji}=\rho_t(j) (\mathbb{E}( _j X_t(i))-\delta_{ij}).$ 
 Similarly, if $\mathbb{E}\left( || _jX_t||^p\right)$
exists and depends continuously on $t$,
then the moments of $\{Z_t\}$
of order $p$ satisfy inhomogeneous linear equations, and if $\mathbb{E}\left( || _jX_t||^p\right)$
is uniformly bounded in both $t$ and $j$, then the coefficients of those equations are uniformly bounded.
In particular, Assumption~3$'$ implies that Assumption~3 is satisfied,
while a bound on the third moment of $\|_jX_t\|$  (see Assumption 4$'$ below) 
would imply that Assumption 4 is satisfied. 

Recall that, in the notation of Section~\ref{nore} applied to the process observed at integer time points,
\[
\mathbb{E} Z_n = M_n^T  Z_0.
\]
Therefore, by part (e) of Proposition~\ref{Propabc}, for each initial population, there is a positive constant $C$ such that
\[
\frac{1}{C} \Lambda_n \leq \| \mathbb{E} Z_n \| \leq C \Lambda_n.
\] 
From (\ref{hjhjl}) it follows that there is a positive constant $c$ such that
\[
\frac{1}{c} \| \mathbb{E} Z_n \| \leq \| \mathbb{E} Z_t \| \leq c \| \mathbb{E} Z_n \|,~~~~n \leq t \leq n+1. 
\]
Therefore, the condition 
$\sum_{k=1}^{\infty}({1}/{\Lambda}_{k}) = \infty$ used in Theorem~\ref{yth} is equivalent to the following:
\begin{equation} \label{ooii}
\int_0^\infty \frac{1}{ \| \mathbb{E} Z_t \|} dt = \infty.
\end{equation}
Thus we have the following continuous time analogue of Theorem~\ref{yth}.
\begin{thm} \label{ythtwo}
Under Assumptions 0$'$-3$'$,  if extinction  of the process $\{Z_t\}$ occurs with probability one for some initial population, then
(\ref{ooii}) holds. If (\ref{ooii}) holds, then extinction with probability one occurs for every initial population. 
\end{thm}

To formulate the next theorem, we will make use of the following assumptions:
\\

4$'$. $\mathbb{E}(\|_jX_t\|^3 \big) \leq K_0$ for some $K_0 > 0$.
\\

5$'$. $\mathbb{P} (Z_t \neq \mathbf{0}) \rightarrow 0$ as $t \rightarrow \infty$. 
\\

6$'$. $\mathbb{E} \| Z_t  \| /\mathbb{P} (Z_t \neq \mathbf{0}) \rightarrow \infty $ as $t \rightarrow \infty$.
\\

Note that if Assumptions 0$'$-6$'$ are satisfied, then Assumptions 1-6 are satisfied by the discrete time process $\{Z_n\}$.
Let $\zeta_t = (\zeta_t(1),...,\zeta_t(d))$ be the random vector obtained from $Z_t$ by conditioning on the event that $Z_t \neq \mathbf{0}$. The following theorem
is an easy consequence of Theorem~\ref{the2}. The proof is left to the reader.  

\begin{thm} \label{the2rr}
Under Assumptions 0$'$-6$'$, for each initial population of the branching process and each vector $u$ with positive components,  we have the following limit in distribution
\begin{equation} \label{maineqrr}
\frac{\langle \zeta_t,u \rangle}{  \mathbb{E} \langle \zeta_t,u \rangle } \rightarrow \xi, ~~{\it as}~t \rightarrow \infty,
\end{equation}
where $\xi$ is an exponential random variable with parameter one. Moreover, if  Assumptions~0$'$-5$'$  are satisfied and, for some initial population, 
the limit in (\ref{maineqrr}) is as specified, then Assumption~6$'$ is  also satisfied. 
\end{thm} 

\appendix
\section{Proof of Proposition \ref{Propabc} }
\label{AppPF}
Let $\cK$ be the cone of positive vectors. Given $u,v\in \cK$, their Hilbert metric distance is defined by
$$ d(u,v)=\ln \frac{\beta(u,v)}{\alpha(u,v)},~~ \text{ where }~
\beta(u,v)=\max_i \frac{v(i)}{u(i)}, \quad
\alpha(u,v)=\min_i \frac{v(i)}{u(i)}. $$
Note that $d$ defines the distance on the space of lines in $\cK$ in the sense that
$$ d(au, bv)=d(u,v),\quad d(u, cu)=0 .$$
Moreover the following estimate holds.

\begin{lem} \label{LmHil-Norm}
{\rm(see, e.g., \cite[Lemma 1.3]{L})}
If $\|u\|=\|v\|=1$,  then
$$ \|u-v\|\leq e^{d(u,v)}-1. $$
\end{lem}

We will also use the following result of G. Birkhoff.

\begin{lem} \label{LmHilbContr} {\rm
(see \cite[Theorem XVI.3.3]{B} or \cite[Theorem 1.1]{L})}
If $ A$ is a linear operator that maps $\cK$ into itself so that $ A(\cK)$ has finite
diameter $\Delta$ with respect to the Hilbert metric, then for each $u,v\in \cK$
$$ \frac{d( A u,  A v)}{d(u,v)}\leq \tanh\left(\frac{\Delta}{4}\right)< 1. $$
\end{lem}

\begin{proof}[Proof of Proposition \ref{Propabc}]
Assumptions 1-3 imply that
$ A_n(\cK)\subset \brcK(R)$, where $R=\sqrt{K_0}/\eps_0$ and
$$ \brcK(R):=\{u: u(i)>0 \text{ for each } i \text{ and } \max_i u(i)\leq R \min_i u(i)\}. $$
Note that if $u, v\in \brcK(R)$, then multiplying these vectors by
$c_u = (\max_i u(i))^{-1}$ and $c_v =(\max_i v(i))^{-1}$, respectively, we get
$$ \beta(u,v)\leq R, \quad \alpha(u,v)\leq \frac{1}{R} $$
and so ${\rm diam}(\brcK(R)) \leq 2\ln R.$

Now let
$\cK_{k,n}=M_{k,n} \cK$ and let $\bbK_{k,n}$ denote the set of elements of $\cK_{k,n}$ with unit norm.
Then, for each fixed $k$, $\bbK_{k,n}$ is a nested sequence of compact sets, 
and Lemma 
\ref{LmHilbContr} shows
that the diameter of $\bbK_{k,n}$ with respect to the Hilbert metric is less then
$(2\ln R) (\tanh({\ln R}/{2}))^{n-k-1}.$
Hence Lemma \ref{LmHil-Norm} shows that
$\cap_{n>k} \bbK_{k,n}$
is a single point, which we call $v_k.$
Since $ A_{k-1} (\cap_{n>k} \bbK_{k,n})=\cap_{n>k-1} \bbK_{k-1,n}$, it follows that
$ A_{k-1} v_k=\lambda_{k-1} v_{k-1}$ for some $\lambda_{k-1}>0.$

Next, let $u_0$ be an arbitrary vector with
$$ \|u_0\|=1,~~ u_0(i)>\eps_0 \text{ for each } i. $$
Let $u_n={M_n^T u_0}/{\|M_n^T u_0\|},$
$\tlambda_n=\| A_n^T u_n\|.$ Note that $u_n\in\brcK(R).$

Then $\{u_n\}$ and $\{v_n\}$ satisfy statements (a)--(e) of Proposition \ref{Propabc}.
Indeed, (a) holds by construction. (b) holds since for each vector $w$ in $\brcK(R)$ of unit norm
$$ \min_i w(i)\geq \frac{\max_i w(i)}{R}\geq \frac{1}{dR} . $$
(c) holds since each entry of $u_n(i)$ and $v_n(i)$ is squeezed between 
$1/R$ and $1$ while each entry of $ A_n$ is between $\eps_0$ and $\sqrt{K_0}.$

We prove the first inequality of part (d), the second is similar.
By Lemma \ref{LmHilbContr}, 
$$ d(M_{n, n+k} v, v_n)\leq \eps_k:=2 (\ln R) \left(\tanh \frac{\ln R}{2}\right)^{k-1}. $$
Note that $\eps_k$ can be made as close to $0$ as we wish by taking $k$ large.
By the definition of the Hilbert metric, there is a number $a_{n,k}$ such that
$$ a_{n,k} v_n \leq \frac{M_{n, n+k} v}{\|M_{n, n+k} v\|} \leq a_{n,k} e^{\eps_k} v_n. $$
Taking the norm, we see that $e^{-\eps_k} \leq a_{n,k}\leq 1. $ This proves part (d) for $k'$ such
that $e^{\eps_{k'}}\leq 1+\delta.$ 

Next,
$$ \langle u_n, v_n\rangle =\left\langle \frac{M_{k,n}^T u_k}{\tLambda_n/\tLambda_{k}}, v_n\right\rangle=
\frac{1}{\tLambda_n/\tLambda_{k}} \langle u_k, M_{k,n} v_n \rangle =
\frac{\Lambda_n/\Lambda_{k}}{\tLambda_n/\tLambda_{k}} \langle u_k, v_k \rangle .$$
Due to parts (a) and (b) proved above,  $\langle u_j, v_j\rangle$ are uniformly bounded from above and below, 
i.e., $\eps_0 d\leq \langle u_j, v_j \rangle\leq 1,$ proving the first inequality of part (e).
To prove the second inequality, we note that by the foregoing discussion
there is a constant $L$ such that for each $j$ and $n$ we have
$$ \frac{1}{L} v_{n-1} \leq  A_{n-1} e_j \leq L v_{n-1}. $$
Applying $M_{k, n-1}$ to this inequality and using that
$M_{k, n-1} v_{n-1}=\Lambda_{n-1}/\Lambda_k v_k$, we get
$$ \frac{v_k(i)}{L} \leq \frac{M_{k,n} (i,j)}{\Lambda_{n-1}/\Lambda_k}\leq L v_{k}(i). $$
Combining this with parts (b) and (c) established above, we obtain the second inequality of part (e).  
The proof of Proposition \ref{Propabc} is complete. 
\end{proof}

\section{Skipping generations}
\label{AppSkip}

\begin{proof}[Proof of Proposition \ref{PrRarify}]
(a) If Assumption 1 is satisfied, then the probability to survive till
time $l(n+1)-1$ starting from a single particle at time $ln$ is bounded from below. One of the surviving particles will have two or more offspring of type $i$ with probability bounded from below.

To prove parts (b) and (c), let us consider $l=2$. Then the result for larger $l$ follows similarly by
induction since particles of generation $l+1$ are children of particles of generation $l.$ To prove
part (b), it suffices to show
that $\mathbb{E}((Z_{n+2}(i;k))^2|Z_n=e_j)\leq \bar{K}$, where $Z_{n+2}(i;k)$ is the number of
particles of type $i$ at time $n+2$ whose parents have type $k$.
In other words, it suffices to bound $\EXP(Y^2)$, where $Y=\sum_{m=1}^N X_m $, $X_m$ are independent, have common distribution $\cX$, and are independent of the random variable $N$,
where also $\EXP(\cX^2)\leq K_1,$ $\EXP(N^2)\leq K_2$. Note that
\[
\EXP(Y^2) = \EXP(N\EXP(\cX^2) + (N(N-1)/2)(\EXP(\cX))^2),
\] which gives the desired bound. 

Likewise, to prove (c) it suffices to show that the random variables $Y_n^2$ are uniformly integrable, where 
$ Y_n=\sum_{m=1}^{N_n} X_{m,n} $, $X_{m,n}$ are independent, have common distribution $\cX_n$ and are independent of the random variable $N_n$,
and $\cX_n^2,N_n^2$ are uniformly integrable. We have
$$ Y_n^2=Y_n^2 \chi_{\{N_n >M\}}+Y_n^2 \chi_{\{N_n \leq M\}}. $$
The expectation of the first term equals 
\[
\EXP\left(Y_n^2 \chi_{\{N_n >M\}}\right) = \EXP\left(N_n\EXP(\cX_n^2)\chi_{\{N_n >M\}} + 
\frac{N_n(N_n-1}{2}\chi_{\{N_n >M\}}(\EXP(\cX_n))^2\right).
\]
This expression can be made arbitrarily small by choosing a sufficiently large $M$
since $\EXP(\cX_n^2)$ are uniformly bounded and $N_n^2$ are uniformly integrable. 
For the second term we have

\[
Y_n^2 \chi_{\{N_n \leq M\}} \leq \left(\sum_{m=1}^M X_{m,n}\right)^2,
\]
which is uniformly integrable due to the uniform integrability of $\cX^2_n.$

It remains to establish (d). Choose $l$ so that
\begin{equation}
\label{LargeExp}  
\left(1+\frac{\eps_0}{2}\right)^l>\fb.
\end{equation}
It suffices to show that for each $j$ 
\begin{equation}\label{lsteps}
 \PROB(_jZ_l=\mathbf{0})\geq \eps_1, 
 \end{equation}
where $\eps_1$ depends only on $\eps_0$ and $\fb.$

Given $j\in S$ and $0 \leq n \leq l$, we
will say that $(n,j)$ is 1-unstable if $n < l$ and
$$ \PROB(Z_{n+1}=\mathbf{0}|Z_n=e_j)\geq \frac{\eps_0}{2}. $$
Otherwise we will say that $(n,j)$ is 1-stable.

For $p>0$, we say that
$(n,j)$ is $(p+1)$-unstable if it is either $p-$unstable or $n < l - 1$ and 
$$ \PROB(Z_{n+1}(m)=0\;\; \forall m: (n+1, m) \text{ is }p\text{-stable}|Z_n=e_j)\geq \frac{\eps_0}{2}. $$

Otherwise we will say that $(n,j)$ is $(p+1)$-stable. For example $(n,j)$ is 2-unstable if 
either it is 1-unstable, or, with probability which is not too small, all its children are 1-unstable.
 
We call $l$-stable pairs simply stable. A particle from generation $n$ of type $j$ will be called stable if the pair $(n,j)$ is stable. We claim that all generation $0$ particles are unstable. Indeed,
by definition, each stable particle has at least one stable child with probability at least $1-\frac{\eps_0}{2}$,
and, by Assumption 1, it has at least two stable children with probability at least $\eps_0.$ Accordingly, for each
stable particle, the expected number of its stable children is at least $1+\frac{\eps_0}{2}.$
Hence, had $(0,j)$ been stable, the expected number of its (stable) decedents after $l$ generations would have been
greater than $(1+\frac{\eps_0}{2})^l\geq \fb$ contradicting~\eqref{NotSuperCrit}. 

Set $M = {4\fb}/{\eps_0}$. Then, from (\ref{NotSuperCrit}), it follows that for each $j \in S$ and $n \geq 0$,
\[
\PROB(|Z_{n+1}| \geq M | Z_n = e_j) \leq \frac{\eps_0}{4}.
\] 
Define $\eta_p$ as follows
\[
\eta_p = \inf \PROB(Z_{n+p}=\mathbf{0}|Z_n=e_j),
\]
where the infimum is taken over all $(n,j)$ which are $p-$unstable. Note that \[
\eta_1 \geq \frac{\eps_0}{2}~~~ \text{and}~~~ \eta_p \geq \frac{\eps_0}{4} \eta_{p-1}^M,
\]
where the factor ${\eps_0}/{4}$ represents the probability that the original particle had fewer than $M$ children all of which are $(p-1) -$unstable and $\eta_{p-1}^M$ is the probability that all these children leave no descendants after $p-1$ steps. This proves $(\ref{lsteps})$ with $l$ given by $(\ref{LargeExp})$ and $\eps_1 = \eta_l$.
\end{proof}


\acks  While working on this
article,  all authors were partially supported by UMD REU grant
DMS-1359307. In addition,
D. Dolgopyat was supported by NSF grant   DMS-1362064   
and L. Koralov was supported by NSF grant DMS-1309084  and ARO grant W911NF1710419.

\end{document}